\documentclass[11pt]{article}
\usepackage{exscale}
\usepackage{tabularx}
\usepackage[algosection,ruled,lined,linesnumbered,longend]{algorithm2e}
\usepackage{multirow}
\usepackage{mathrsfs} 
\usepackage{algorithmic}
\usepackage{amssymb}
\usepackage{amsmath}
\usepackage{amsthm}
\usepackage{latexsym}
\usepackage[dvips]{graphicx,color,epsfig}
\usepackage{fullpage}        
\usepackage{float}
\usepackage{cite}
\usepackage{verbatim}
\usepackage{ifthen}
\usepackage[bookmarks,pagebackref,
    pdfpagelabels=true, 
    ]{hyperref}
\usepackage{makeidx}
\makeindex

\numberwithin{equation}{section}   

\definecolor{rblue}{rgb}{.255,.41,.884} 

\newcommand{\integers}{\Z}

\newcommand{\textdef}[1]{\textit{#1}\index{#1}}
\newtheorem{theorem}{Theorem}[section]
\newtheorem{lemma}[theorem]{Lemma}
\newtheorem{definition}[theorem]{Definition}

\newtheorem{prop}[theorem]{Proposition}

\newtheorem{lem}[theorem]{Lemma}

\newtheorem{remark}[theorem]{Remark}

\def\eqref#1{{\normalfont(\ref{#1})}}

\def\R{\mathbb{R}}

\def\Z{\mathbb{Z}}

\def\Rn{\mathbb{R}^n}

\def\Ss{\mathcal{S}}
\newcommand{\Sn}[1][]{\mathcal{S}^{\ifthenelse{\equal{#1}{}}{n}{#1}}}
\newcommand{\Sk}[1][]{\mathcal{S}^{\ifthenelse{\equal{#1}{}}{k}{#1}}}
\newcommand{\Snp}[1][]{\mathcal{S}_+^{\ifthenelse{\equal{#1}{}}{n}{#1}}}
\newcommand{\Skp}[1][]{\mathcal{S}_+^{\ifthenelse{\equal{#1}{}}{k}{#1}}}
\newcommand{\Snpp}[1][]{\mathcal{S}_{++}^{\ifthenelse{\equal{#1}{}}{n}{#1}}}

\newcommand{\EE}{{\mathcal E} }

\newcommand{\DD}{{\mathcal D} }

\newcommand{\GG}{{\mathcal G} }
\newcommand{\OO}{{\mathcal O} }

\newcommand{\TT}{{\mathcal T} }

\newcommand{\ZZ}{{\mathcal Z} }

\newcommand{\LP}{{\bf{\rm LP\,}}}

\DeclareMathOperator{\kvec}{{vec}}

\DeclareMathOperator{\trace}{{trace}}
\DeclareMathOperator{\ext}{{ext}}

\DeclareMathOperator{\cut}{{cut}}

\DeclareMathOperator{\tr}{{trace}}
\DeclareMathOperator{\Diag}{{Diag}}
\DeclareMathOperator{\diag}{{diag}}

\DeclareMathOperator{\Mat}{{Mat}}

\DeclareMathOperator{\relint}{{relint}}

\newcommand{\Mcal}{{\mathcal M}}

\newcommand{\NN}{{\mathcal N}}

\newcommand{\A}{{\mathcal A}}

\newcommand{\MM}{{\mathcal M}}
\newcommand{\QQ}{{\mathcal Q}}
\newcommand{\bem}{\begin{pmatrix}}
\newcommand{\bemp}{\begin{pmatrix}}
\newcommand{\bemb}{\begin{bmatrix}}
\newcommand{\eem}{\end{pmatrix}}
\newcommand{\eemp}{\end{pmatrix}}
\newcommand{\eemb}{\end{bmatrix}}
\newcommand{\beq}{\begin{equation}}
\newcommand{\bet}{\begin{table}}
\newcommand{\eeq}{\end{equation}}
\newcommand{\beqr}{\begin{eqnarray}}

\newcommand{\bpr}{\begin{proof}}

\newcommand{\epr}{\qquad \end{proof}}

\newcommand{\nc}{\newcommand}
\nc{\arrow}{{\rm arrow\,}}
\nc{\Arrow}{{\rm Arrow\,}}
\nc{\BoDiag}{{\rm B^0Diag\,}}
\nc{\bodiag}{{\rm b^0diag\,}}

\nc{\Mmn}{{\mathcal M}_{m,n} }
\nc{\kwqqp}{Q{$^2$}P\,}
\nc{\kwqqps}{Q{$^2$}Ps}

\nc{\notinaho}{(X,S)\in \overline{AHO}(\A)}
\nc{\inaho}{(X,S)\in AHO(\A)}
\nc{\In}{\frac 1{\sqrt n} I}

\newcommand{\bea}{\begin{eqnarray}}%
\newcommand{\eea}{\end{eqnarray}}%
\newcommand{\beas}{\begin{eqnarray*}}%
\newcommand{\eeas}{\end{eqnarray*}}%
\newcommand{\Rnk}{\R^{n \times k}}%
\renewcommand{\S}{\mathcal{S}}

{}

\title{
Eigenvalue, Quadratic Programming,
 and Semidefinite Programming Relaxations\\
      for\\
a Cut Minimization Problem
\footnote{Presented at Retrospective Workshop on
Discrete Geometry, Optimization and Symmetry, November 24-29, 2013,
Fields Institute, Toronto, Canada.}
}

\author{
\href{http://www.cs.ubc.ca/~tkpong/}
  {Ting Kei Pong}
        \thanks{Department of Applied Mathematics,
        the Hong Kong Polytechnic University, Hung Hom, Hong Kong.
        This author was also supported as a PIMS postdoctoral fellow at Department of Computer Science, University of British Columbia, Vancouver, during the early stage of the preparation of the manuscript.
        Email: {tk.pong@polyu.edu.hk}
                }
\and
  {Hao Sun}
        \thanks{Department of Combinatorics and Optimization,
                University of Waterloo, Ontario N2L 3G1, Canada.
              Research supported by an Undergraduate Student Research
Award from The Natural Sciences and Engineering
                Research Council of Canada.
        Email: {hao\_sun@live.com}
                }
\and
\href{http://orion.math.uwaterloo.ca/~spschurr/}
{Ningchuan Wang}
        \thanks{
              Research supported by The Natural Sciences and Engineering
                Research Council of Canada
 and by the U.S. Air Force Office of Scientific Research.
                Email: {wangningchuan1987@hotmail.com}}
\and
       \href{http://orion.math.uwaterloo.ca/~hwolkowi/}
{Henry Wolkowicz}
        \thanks{Research supported in part by The Natural Sciences and Engineering
                Research Council of Canada
 and by the U.S. Air Force
Office of Scientific Research.
                Email: {hwolkowicz@uwaterloo.ca}}
}
\date{\today}

\begin{document}
\bibliographystyle{plain}
          \maketitle
\begin{center}
\href{http://www.uwaterloo.ca/}
{University of Waterloo}\\
\href{http://www.math.uwaterloo.ca/CandO_Dept/homepage.html}
{Department of Combinatorics \& Optimization}\\
          Waterloo, Ontario N2L 3G1, Canada\\
          Research Report\\
\end{center}





\begin{abstract}
We consider the problem of partitioning the node set of a graph into $k$
sets of given sizes in order to \emph{minimize the cut} obtained using
(removing) the $k$-th set. If the resulting cut has value $0$, then we
have obtained a vertex separator.
This problem is closely related to the graph
partitioning problem. In fact, the model we use is the same as that for
the graph partitioning problem except for a different
\emph{quadratic} objective function.
We look at known and new bounds obtained from various relaxations for
this NP-hard problem. This includes: the standard eigenvalue bound,
projected eigenvalue bounds using both the adjacency matrix and the
Laplacian, quadratic programming (QP) bounds based on recent successful
QP bounds for the quadratic assignment problems,
and semidefinite programming bounds. We include numerical tests for
large and \emph{huge} problems that illustrate the efficiency of the
bounds in terms of strength and time.

\end{abstract}

\tableofcontents
\listoftables
\listoffigures

\section{Introduction}
\label{sect:intro}
We consider a special type of \textdef{minimum cut problem, MC}.
\index{MC, minimum cut problem}
The problem consists in partitioning the node set of a graph into $k$
sets of given sizes in order to \emph{minimize the cut} obtained by
removing the $k$-th set. This is achieved by minimizing the
number of edges connecting distinct sets after removing the $k$-th set, as described in \cite{ReLiPi:13}.
This problem arises when finding a re-ordering to bring the sparsity pattern of a large sparse positive definite matrix into a block-arrow shape so as to minimize fill-in in its Cholesky factorization.
The problem also arises as a subproblem of the \textdef{vertex separator
problem, VS}.\index{VS, vertex separator problem}
In more detail, a vertex separator is a
set of vertices whose removal from the graph results in a disconnected
graph with $k-1$ components. A typical VS problem
has $k=3$ on a graph with $n$ nodes, and it seeks a vertex separator which is optimal subject to some constraints on the
partition size.
This problem can be solved
by solving an MC for each possible partition size. Since there are at most
$\binom{n-1}{2}$ $3$-tuple integers that sum up to $n$, and it is known that
VS is NP-hard in general \cite{RHLewis:14,ReLiPi:13}, we see that MC is also NP-hard
when $k\ge 3$.


Our MC problem is closely related to the
\textdef{graph partitioning problem, GP}, which is also NP-hard; see the
discussions in \cite{RHLewis:14}.
In both problems one can use a model
with a \emph{quadratic} objective function over the set of
\textdef{partition matrices}.  The model we use is the same as that for
GP except that the quadratic objective function is different.
We study both existing and new bounds and provide both theoretical
properties and empirical results. Specifically, we adapt and improve known techniques for deriving lower bounds for GP to derive bounds for MC.
We consider eigenvalue bounds, a convex quadratic programming, QP, lower bound, as well as lower bounds based on semidefinite programming, SDP,\index{semidefinite programmming, SDP}
\index{SDP, semidefinite programmming} relaxations.

%
%

We follow the approaches in \cite{HaReWo:89,ReLiPi:13,ReWo:90}
for the eigenvalue bounds. In particular, we replace the standard
quadratic objective function for GP,
e.g.,~\cite{HaReWo:89,ReWo:90} with that used in \cite{ReLiPi:13} for
MC. It is shown in \cite{ReLiPi:13} that
one can equally use either the adjacency matrix $A$ or the negative
Laplacian $(-L)$ in the objective function of the model.
We show in fact that one can use $A-\Diag(d), \forall d\in \Rn$, in the
\index{$\Diag$}
model, where $\Diag(d)$ denotes the diagonal matrix with diagonal $d$.
However, we emphasize and show that this is no longer true for the
eigenvalue bounds and that using $d=0$ is, empirically, stronger.
Dependence of the eigenvalue lower bound on diagonal perturbations was
also observed for the quadratic assignment problem, QAP, and GP,
see e.g.,~\cite{ReWo:89,FaReWo:92}.
In addition, we find a new projected eigenvalue lower bound using $A$ that has
three terms that can be found explicitly and efficiently. We illustrate
this empirically on large and huge scale sparse problems.

Next, we extend the approach in
\cite{AnsBrix:99,AnWo:98,BrixiusAnstr:01} from
the QAP to MC.
This allows for a QP bound that is based on
SDP duality and that can be solved
efficiently. The discussion and derivation of this lower bound is new
even in the context of GP.
Finally, we follow and extend the
approach in \cite{WoZh:96} and derive and test SDP relaxations. In
particular, we answer a question posed in \cite{WoZh:96} about redundant
constraints. This new result simplifies the SDP relaxations even in the
context of GP.


\subsection{Outline}
We continue in Section \ref{sect:prel} with
preliminary descriptions and results on our special MC.
This follows the approach in \cite{ReLiPi:13}.
In Section \ref{sect:eigs} we outline the basic eigenvalue bounds
and then the projected eigenvalue bounds following
the approach in \cite{HaReWo:89,ReWo:90}. Theorem \ref{thm:projeig}
includes the projected bounds along with our new three part eigenvalue
bound. The three part bound can be calculated explicitly and efficiently
by finding $k-1$ eigenvalues and a minimal scalar product, and making
use of the result in Section \ref{sect:expllin}.
The QP bound is described in Section
\ref{sect:qpbnd}. The SDP bounds are presented in Section \ref{sect:sdpbnd}.

Upper bounds using feasible solutions are given in Section
\ref{sect:feassoln}. Our numerical tests are in Section
\ref{sect:numertests}. Our concluding remarks are in Section \ref{sect:concl}.

\section{Preliminaries}
\label{sect:prel}
We are given an undirected graph ${\sf G}=(N,E)$ with
\index{graph! ${\sf G}$}
\index{${\sf G}$, graph}
a nonempty node set $N=\{1,\ldots,n\}$ and
\index{$N=\{1,\ldots,n\}$}
\index{graph! node set, $N=N({\sf G})$}
a nonempty edge set $E$.
In addition, we have a positive integer vector of set
\index{graph! edge set, $E=E({\sf G})$}
sizes $m=(m_1,\ldots,m_k)^T \in \integers^k_{+}$, $k>2$,
\index{$m$, set sizes}
\index{set sizes, $m$}
such that the sum of the components $m^Te=n$. Here $e$ is the vector of ones
of appropriate size.
Further, we let $\Diag(v)$ denote the diagonal matrix
formed using the vector $v$;
the adjoint $\diag(Y)=\Diag^*(Y)$ is the vector formed from the diagonal
of the square matrix $Y$. We let $\ext(K)$ represent the extreme
points of a convex set $K$.
We let $x=\kvec (X)\in \R^{nk}$ denote the vector formed
(columnwise) from the matrix $X$; the adjoint and inverse is $\Mat(x)\in
\Rnk$. We also let $A \otimes B$ denote the Kronecker product; and
$A \circ B$ denote the Hadamard product.
\index{$A \circ B$, Hadamard product}
\index{Hadamard product, $A \circ B$}
\index{$A \otimes B$, Kronecker product}
\index{Kronecker product, $A \otimes B$}
\index{$\Rnk$, $n\times k$ matrices}
\index{$\kvec(X)$, vector from matrix}
\index{vector from matrix, $\kvec(X)$}
\index{$\Mat(x)$, matrix from vector}
\index{matrix from vector, $\Mat(x)$}
\index{$\cdot^*$, adjoint}
\index{adjoint, $\cdot^*$}
\index{$ext$, extreme points}
\index{extreme points, $ext$}
\index{$\Diag$}
\index{$\diag$}
\index{$e$, vector of ones}
\index{vector of ones, $e$}

We let
\[
P_m:=\left\{ (S_1,\ldots,S_k): S_i\subset N, |S_i|=m_i, \forall i,
S_i\cap S_j=\emptyset,
\text{ for } i\neq j, \cup_{i=1}^k S_i=N
   \right\}
\]
denote the set of all \emph{partitions of $N$} with the appropriate sizes
specified by $m$.
The partitioning is encoded using an $n\times k$
\emph{partition matrix} $X\in
\R^{n\times k}$ where the column $X_{:j}$ is the incidence vector for
the set $S_j$
\index{partitions}
\index{$P_m$, set of all partitions}
\index{partitions! $P_m$, set of all partitions}
\index{partitions! set of all partitions, $P_m$}
\index{partitions! partition matrix, $X$}
\[
X_{ij}=\left\{
\begin{array}{cc}
1 & \text{if } i \in S_j \\
0 & \text{otherwise}.
\end{array}
\right.
\]
Therefore, the set cardinality constraints are given by $X^Te=m$; while the
constraints that each vertex appears in exactly one set is given by
$Xe=e$.
\index{partitions! set of all partition matrices, $\Mcal_m$}
\index{$\Mcal_m$, set of all partition matrices}

The set of partition matrices can be represented using various linear and quadratic constraints. We
present several in the following. In particular, we phrase
the linear equality constraints as quadratics for use in the Lagrangian
relaxation below in Section \ref{sect:sdpbnd}.
\begin{definition}
We denote the set of zero-one, nonnegative, linear equalities,
doubly stochastic type,
$m$-diagonal orthogonality type,
$e$-diagonal orthogonality type,
and gangster constraints as, respectively,
\index{constraints}
\index{constraints! $\DD_O$, m-diag. orthogonality}
\index{constraints! $\DD_e$, e-diag. orthogonality}
\index{constraints! $\DD$, doubly stochastic type}
\index{constraints! $\EE$, linear equalities}
\index{constraints! $\ZZ$, zero-one}
\index{constraints! $\NN$, nonnegativity}
\index{constraints! $\GG$, gangster set}
\[
\begin{array}{rcl}
\ZZ  &:=& \{ X \in \R^{n\times k}: X_{ij}\in \{0,1\}, \forall ij\}
         =\{ X \in \R^{n\times k}: \left(X_{ij}\right)^2= X_{ij}, \forall ij\}
           \\
\NN  &:=& \{ X \in \R^{n\times k}: X_{ij}\geq  0, \forall ij\}   \\
\EE &:=& \{ X \in \R^{n\times k}: Xe=e,X^Te=m\}
     = \{ X \in \R^{n\times k}: \|Xe-e\|^2+\|X^Te-m\|^2=0\}
 \\
\DD &:=& \{ X \in \R^{n\times k}: X\in \EE \cap \NN \}   \\
\DD_O &:=& \{ X \in \R^{n\times k}: X^TX=\Diag(m)\}   \\
\DD_e &:=& \{ X \in \R^{n\times k}: \diag(XX^T)=e\}   \\
\GG &:=& \{ X \in \R^{n\times k}: X_{:i}\circ X_{:j} =0, \forall i\neq j \}   \\
\end{array}
\]
\end{definition}

There are many equivalent ways of representing the set of all
partition matrices. Following are a few.
\begin{prop}
\label{prop:partmatrices}
The set of partition matrices in $\Rnk$ can be expressed as the
following.
\begin{equation}
\label{eq:partmatrices}
\begin{array}{rcl}
\Mcal_m
&=& \EE \cap \ZZ
\\&=& \ext(\DD)
\\&=&  \EE \cap \DD_O \cap \NN 
\\&=&  \EE \cap \DD_O \cap \DD_e \cap \NN 
\\&=&  \EE \cap \ZZ \cap \DD_O \cap \GG \cap \NN.
\end{array}
\end{equation}
\end{prop}
\begin{proof}
The first equality follows immediately from the definitions.
The second equality follows from the transportation type constraints and
is a simple consequence of Birkhoff and Von Neumann theorems
that the extreme
points of the set of doubly stochastic matrices are the permutation
matrices, see e.g.,~\cite{MR874114}.  The third equality
is shown in \cite[Prop. 1]{ReLiPi:13}. The fourth and fifth
equivalences contain redundant sets of constraints.
\end{proof}

We let $\delta(S_i,S_j)$ denote the set of edges between the sets of nodes
\index{$\delta(S_i,S_j)$, set of edges between $S_i,S_j$} $S_i,S_j$, and
we denote the set of edges with endpoints in distinct partition sets
$S_1, \ldots, S_{k-1}$ by
\begin{equation}
\label{eq:deltaS}
\delta(S)= \cup_{i<j<k} \delta(S_i,S_j).
\end{equation}
The minimum of the cardinality $|\delta(S)|$ is denoted
\begin{equation}
\label{eq:cutS}
\cut(m) = \min \{ |\delta(S)| : S \in P_m \}.
\index{$\cut(S)$}
\end{equation}
The graph ${\sf G}$ has a \emph{vertex separator} if there exists an
$S\in P_m$ such that the removal of set $S_k$ results in the sets
$S_1, \ldots, S_{k-1}$ being pairwise disjoint. This is equivalent to
\index{vertex separator, VS}
\index{VS, vertex separator}
$\delta(S)=\emptyset$, i.e.,~$\cut(m)=0$.
Otherwise, $\cut(m)>0$.\footnote{A discussion of the relationship of
$\cut(m)$ with the bandwidth of the graph is given in
e.g.,~\cite{ReLiPi:13,MR2363705,MR3060945}. Particularly, for $k=3$, if $\cut(m)>0$, then $m_3 + 1$ is a lower bound for the bandwidth.}

We define the $k \times k$ matrix
\index{$\S^k$, symmetric matrices}
\index{symmetric matrices, $\S^k$}
\index{trace inner-product}
\[
B:=
\begin{bmatrix}
ee^T -I_{k-1} & 0 \cr
0 & 0
\end{bmatrix} \in \S^k,
\]
where
$\S^k$ denotes the vector space of $k\times k$ symmetric matrices
equipped with the trace inner-product, $\langle S,T \rangle = \trace ST$.
We let $A$ denote the adjacency matrix of the graph and let
$L:=\Diag(Ae)-A$ be the Laplacian.

In \cite[Prop. 2]{ReLiPi:13}, it was shown that $|\delta(S)|$ can be represented in terms of a quadratic function of
the partition matrix $X$, i.e.,~as
$\frac 12 \trace (-L)XBX^T$ and $\frac 12 \trace AXBX^T$, where we note that
the two matrices $A$ and $-L$ differ only on the diagonal.
From their proof, it is not hard to see that their result can be slightly extended as follows.
\begin{prop}
\label{prop:partmatrix}
For a partition $S\in P_m$, let $X\in \MM_m$ be the associated
partition matrix. Then
\index{graph! Laplacian matrix, $L$}
\index{graph! $L$, Laplacian matrix}
\index{graph! adjacency matrix, $A$}
\index{$A$, adjacency matrix}
\index{graph! $A$, adjacency matrix}
\begin{equation}
\label{eq:objfn}
|\delta(S)|=   \frac 12 \trace \left(A-\Diag
(d)\right)XBX^T, \quad  \forall d\in \Rn.
\end{equation}
In particular, setting $d=0, Ae$, respectively yields $A, -L$.
\end{prop}
\begin{proof}
The result for the choices of $d=0, Ae$,
equivalently $A, -L$, respectively, was proved in
\cite[Prop. 2]{ReLiPi:13}. Moreover, as noted in the proof of
\cite[Prop. 2]{ReLiPi:13}, ${\diag}(XBX^T) = 0$.
Consequently,
\[
\frac 12 \trace AXBX^T = \frac 12 \trace
\left(A-{\Diag}(d)\right)\!XBX^T, \quad \forall d\in \mathbb{R}^n.
\]
\end{proof}

In this paper we focus on the following problem
given by \eqref{eq:cutS} and \eqref{eq:objfn}:
\index{cut minimization problem}
\index{objective function}
\begin{equation}
\label{eq:Prob}
  \begin{array}{rrcl}
    \cut(m)= &\min & \frac 12  \trace (A-\Diag(d))XBX^T \\
    &\text{s.t.} & X\in \Mcal_m;
  \end{array}
\end{equation}
here $d\in \Rn$.
We recall that if $\cut(m)=0$, then
we have obtained a vertex separator, i.e.,~removing the $k$-th set
results in a graph where the first $k-1$ sets are disconnected.
On the other hand, if we find a positive lower bound $\cut(m) \geq \alpha >
0$, then no vertex separator can exist for this $m$. This observation can be employed
in solving some classical vertex separator problems, which look for an
``optimal" vertex separator in the case $k=3$ under constraints on
$(m_1,m_2,m_3)$. Specifically, since there are  at most $\binom{n-1}{2}$
$3$-tuple integers summing up to $n$, one only needs to consider at most
$\binom{n-1}{2}$ different MC problems in order to find the
\emph{optimal} vertex separator.

Though any choice of $d\in \Rn$ is equivalent for \eqref{eq:Prob} on the
feasible set $\Mcal_m$, as we shall see repeatedly throughout the paper, this does \emph{not} mean that they are
equivalent on the relaxations that we look at below.
We would also like to mention that similar observations concerning
diagonal perturbation were previously made for the QAP, the GP and their relaxations, see e.g.,~\cite{ReWo:89,FaReWo:92}.
Finally, note that the feasible set of \eqref{eq:Prob} is the same as
that of the GP, see e.g.,~\cite{ReWo:90,WoZh:96} for the projected eigenvalue
bound and the SDP bound, respectively.  
Thus, the techniques for deriving bounds for MC can
be adapted to obtain new results concerning lower bounds for GP.

\section{Eigenvalue Based Lower Bounds}
\label{sect:eigs}
We now present bounds on $\cut(m)$ based on $X\in \DD_O$, the
$m$-diagonal orthogonality type constraint $X^TX=\Diag(m)$.
For notational simplicity, from now on, we define $M:=\Diag(m)$,
$\tilde m:=
\index{$M=\Diag(m)$}
\index{$\tilde M=\Diag(\tilde m)$}
\index{$\tilde m$}
\left(\sqrt{m_1},\ldots, \sqrt{m_k}\right)^T$ and $\tilde M:= \Diag (\tilde
m)$.
\index{$\tilde B= M^{1/2}BM^{1/2}$}
For a real symmetric matrix $C\in \Ss^t$, we let
\[
\lambda_1(C)\ge \lambda_{2}(C)\ge \cdots \ge \lambda_t(C)
\]
denote the eigenvalues of $C$
in nonincreasing order, and set $\lambda(C)=\left(\lambda_i(C)\right)\in \mathbb{R}^t$.

\subsection{Basic Eigenvalue Lower Bound}

The Hoffman-Wielandt bound \cite{hw53} can be applied to get a simple
eigenvalue bound. In this approach, we solve the relaxed problem
\begin{equation}
\label{eq:eigbndpgm}
\begin{array}{rcc}
\cut(m) \geq &\min & \frac 12  \trace GXBX^T \\
&\text{s.t.}& X\in \DD_O,
\end{array}
\end{equation}
where $G = G(d) =A-\Diag(d)$, $d\in \Rn$. We first introduce the following definition.
\index{$G = A -\Diag(d), d\in \Rn$}
\begin{definition}
\index{$\left\langle x,y \right\rangle_-$, minimal scalar product}
\index{minimal scalar product, $\left\langle x,y \right\rangle_-$}
  For two vectors $x$, $y\in \mathbb{R}^n$, the minimal scalar product is defined by
  \[
  \left\langle x,y\right\rangle_- := \min \left\{ \sum_{i=1}^n x_{\phi(i)}y_i:\; \mbox{$\phi$ is a permutation on $N$} \right\}.
  \]
\end{definition}
We will also need the following two auxiliary results.
\begin{theorem}[{Hoffman and Wielandt \cite{hw53}}]
\label{thm:HWineq}
Let $C$ and $D$ be symmetric matrices of orders $n$ and $k$, respectively, with
$k\leq n$. Then
\begin{equation}\label{eq:1stminprob}
\min \left\{ \trace CXDX^T : X^TX=I_k \right\}
=\left\langle \lambda(C),\begin{pmatrix}\lambda(D) \cr 0
\end{pmatrix}\right\rangle_-.
\end{equation}
The minimum on the left is attained for $X=
\begin{bmatrix}  p_{\phi(1)} & \ldots &
p_{\phi(k)}\end{bmatrix} Q^T$, where $p_{\phi(i)}$ is a normalized eigenvector
to $\lambda_{\phi(i)}(C)$, the columns of
$Q=\begin{bmatrix} q_1 & \ldots & q_k
\end{bmatrix}$ consist of the normalized eigenvectors $q_i$ of
$\lambda_i(D)$, and $\phi$ is the permutation of $\{1,\ldots,n\}$
attaining the minimum in the minimal scalar product.
\qed
\end{theorem}
\begin{lem}[{\cite[Lemma~4]{ReLiPi:13}}]
\label{lem:eigBtilde}
The $k$-ordered eigenvalues of the matrix $\tilde B:= \tilde MB\tilde M$
satisfy
\index{$\tilde B= M^{1/2}BM^{1/2}$}
\[
\lambda_1(\tilde B)>0=
 \lambda_{2}(\tilde B)  > \lambda_{3} (\tilde B)   \geq \ldots \geq
\lambda_{k-1}(\tilde B) \geq \lambda_k(\tilde B).
\qed
\]
\end{lem}
We now present the basic eigenvalue lower bound, which turns out to always be negative.
\begin{theorem}
Let $d\in \Rn$, $G=A-\Diag(d)$. 
Then
\[
\cut(m) \geq 0 >
p^*_{eig}(G):=
  \frac12\left\langle \lambda(G), \begin{pmatrix}
    \lambda({\tilde B})\\ 0
  \end{pmatrix}\right\rangle_-
= \frac12\left(\sum_{i=1}^{k-2} \lambda_{k-i+1}(\tilde B)\lambda_i(G)+
\lambda_{1}(\tilde B)\lambda_{n}(G)\right).
\]
Moreover, the function $p^*_{eig}(G(d))$ is concave as a function of $d\in \Rn$.
\end{theorem}
\begin{proof}
We use the substitution $X=Z\tilde M$, i.e.,~$Z=X\tilde M^{-1}$, in \eqref{eq:eigbndpgm}.
Then the constraint on $X$ implies that $Z^TZ=I$.
We now solve the equivalent problem to \eqref{eq:eigbndpgm}:
\beq
\label{eq:concZ}
\begin{array}{ccc}
\min & \frac 12  \trace GZ(\tilde MB\tilde M)Z^T \\
\text{s.t.}& Z^TZ=I.
\end{array}
\eeq
The optimal value is obtained using the minimal scalar product of
eigenvalues as done in the Hoffman-Wielandt result, Theorem
\ref{thm:HWineq}. From this we conclude immediately that $\cut(m) \geq p^*_{eig}(G)$.
Furthermore, the explicit formula for the minimal scalar product follows immediately from Lemma~\ref{lem:eigBtilde}.

We now show that $p^*_{eig}(G) < 0$.
Note that $\tr \tilde{M}B\tilde{M}= \tr MB =0$.
Thus the sum of the eigenvalues of
$\tilde{B}=\tilde{M}B\tilde{M}$ is 0. Let $\widehat\phi$ be a permutation of $\{1,\ldots,n\}$
that attains the minimum value
$\min\limits_{ \phi \text{ permutation}} \sum_{i=1}^k \lambda_{\phi(i)}(G)\lambda_i(\tilde{B})$. Then for any permutation
$\psi$, we have
\begin{equation}
\label{eq:perms}
\sum_{i=1}^k \lambda_{\psi(i)}(G)\lambda_i(\tilde{B})
\geq \sum_{i=1}^k \lambda_{\widehat{\phi}(i)}(G)\lambda_i(\tilde{B}).
\end{equation}
Now if $\TT$ is the set of all permutations of $\{1,2,\ldots,n\}$, then
we have
\begin{equation}\label{eq:sum}
\sum_{ \psi \in \TT} \left(\sum_{i=1}^k \lambda_{\psi(i)}(G)\lambda_i(\tilde{B})\right) =
\sum_{i=1}^k \left(\sum_{ \psi \in \TT}\lambda_{\psi(i)}(G)\right)\lambda_i(\tilde{B}) =
\left(\sum_{ \psi \in \TT}\lambda_{\psi(1)}(G)\right)\left(\sum_{i=1}^k \lambda_i(\tilde{B})\right) = 0,
\end{equation}
since $\sum_{ \psi \in \TT}\lambda_{\psi(i)}(G)$ is independent of $i$.
This means that there exists at least one permutation $\psi$ so that $\sum_{i=1}^k \lambda_{\psi(i)}(G)\lambda_i(\tilde{B})\le 0$,
which implies that the minimal scalar product must satisfy
$\sum_{i=1}^k
\lambda_{\widehat{\phi}(i)}(G)\lambda_i(\tilde{B}) \leq 0$. Moreover,
in view of \eqref{eq:perms} and \eqref{eq:sum},
this minimal scalar product is
zero if, and only if,
$\sum_{i=1}^k \lambda_{\psi(i)}(G)\lambda_i(\tilde{B}) =0$, for all $\psi
\in \TT$.
Recall from Lemma~\ref{lem:eigBtilde} that $\lambda_1(\tilde{B}) > \lambda_k(\tilde{B})$.
Moreover, if
all eigenvalues of $G$ were equal, then necessarily $G=\beta I$ for some $\beta\in \mathbb{R}$ and $A$
must be diagonal. This implies that $A = 0$, a contradiction. This contradiction shows that $G(d)$
must have at least two distinct eigenvalues, regardless of the choice of $d$.
Therefore, we can change the order and
change the value of the scalar product on the left in \eqref{eq:perms}.
Thus $p^*_{eig}(G)$ is strictly negative.

Finally, the concavity follows by observing from
\eqref{eq:concZ} that
\[
p^*_{eig}(G(d))=
\min_{Z^TZ=I}  \frac 12  \trace G(d)Z(\tilde MB\tilde M)Z^T,
\]
is a function obtained as a minimum of a set of functions affine in $d$,
and recalling that the minimum of affine functions is concave.
\end{proof}

\begin{remark}
We emphasize here that the eigenvalue bounds depend on the choice of
$d\in \Rn$. Though the $d$ is irrelevant
in Proposition \ref{prop:partmatrix}, i.e.,~the function is equivalent on the
feasible set of partition matrices $\Mcal_m$, the values
are no longer equal on the relaxed set $\DD_O$.
Of course the values are negative and not useful as a
bound. We can fix $d = Ae\in \Rn$ and consider the bounds
\[
\cut(m) \geq 0 >
p^*_{eig}(A-\gamma \Diag(d)) =
  \frac12\left\langle \lambda(A-\gamma \Diag(d)), \begin{pmatrix}
    \lambda({\tilde B})\\ 0
  \end{pmatrix}\right\rangle_-, \quad \gamma \geq 0.
\]
From our empirical tests on random problems, we
observed that the maximum occurs for $\gamma$ closer to $0$ than $1$, thus
illustrating why the bound using $G=A$ is better than the one using $G=-L$.
This motivates our use of $G=A$ in the simulations below for the
improved bounds.
\end{remark}

\subsection{Projected Eigenvalue Lower Bounds}
Projected eigenvalue bounds for the
QAP, and for GP are presented and studied in
\cite{FaReWo:92,HaReWo:89,ReWo:90}. They have proven to be surprisingly
stronger than the basic eigenvalue bounds. (Seen to be $<0$ above.)
These are based on a
special parametrization of the affine span of the
linear equality constraints, $\EE$.
Rather than solving for the basic eigenvalue bound using the program in
\eqref{eq:eigbndpgm}, we include the linear equality
constraints $\EE$, i.e.,~we consider the problem
\begin{equation}
\label{eq:projbndX}
\begin{array}{ccl}
\min & \frac 12  \trace GXBX^T \\
\text{s.t.}& X\in \DD_O \cap \EE,
\end{array}
\end{equation}
where $G=A-\Diag(d)$, $d\in \Rn$.

We define the $n\times n$ and $k\times k$ orthogonal matrices $P,Q$ with
\index{orthogonal matrices, $\OO_n$}
\index{$\OO_n$, orthogonal matrices}
\begin{equation}
\label{eq:defPQ}
P=\begin{bmatrix} \frac 1{\sqrt{n}} e & V \end{bmatrix} \in \OO_n,
\quad
Q=\begin{bmatrix} \frac 1{\sqrt{n}} \tilde m  & W \end{bmatrix} \in \OO_k.
\end{equation}
\begin{lemma}{\cite[Lemma 3.1]{ReWo:90}}
\label{lem:paramE}
Let $P,Q,V,W$ be defined in \eqref{eq:defPQ}. Suppose that $X \in
\R^{n\times  k}$ and $Z\in \R^{(n-1)\times  (k-1)}$ are related by
\begin{equation}
\label{eq:XZrelate}
X=P \begin{bmatrix}  1 & 0 \cr 0 & Z \end{bmatrix}  Q^T \tilde M.
\end{equation}
Then the following holds:
\begin{enumerate}
\item
$X\in \EE$.
\item $ X \in \NN \Leftrightarrow  VZW^T \geq -\frac 1n e\tilde m^T $.
\item $ X\in \DD_O \Leftrightarrow   Z^TZ = I_{k-1} $.
\end{enumerate}
Conversely, if $X\in \EE$, then there exists $Z$ such that the
representation \eqref{eq:XZrelate} holds.
\qed
\end{lemma}

Let $\QQ:  \R^{(n-1)\times(k-1)} \rightarrow
\R^{n\times k}$ be the linear transformation defined by
$\QQ(Z)=VZW^T\tilde M$ and define $\widehat X=\frac 1n em^T\in
\R^{n\times k}$. Then $\widehat X \in \EE$, and Lemma~\ref{lem:paramE}
\index{$\widehat X=\frac 1n em^T$}
states that $\QQ$ is an invertible
transformation between $\R^{(n-1)\times(k-1)}$ and $\EE - \widehat X$.
Moreover, from
\eqref{eq:XZrelate}, we see that $X\in \EE$ if, and only if,
\begin{equation}\label{eq:XZrelate2}
\begin{array}{rcl}
X
&=&
P \begin{bmatrix}  1 & 0 \cr 0 & Z \end{bmatrix}  Q^T \tilde M
\\&=&
\begin{bmatrix} \frac e{\sqrt n} & V \end{bmatrix}
\begin{bmatrix} 1 & 0 \cr 0 & Z \end{bmatrix}
\begin{bmatrix} \frac {1}{\sqrt n}\tilde m^T \\ W^T \end{bmatrix} \tilde M
\\&=&
\frac 1n e m^T + VZW^T\tilde M
\\&=&
\widehat X + VZW^T\tilde M,
\end{array}
\end{equation}
for some $Z$.
Thus, the set $\EE$ can be parametrized using $\widehat X + VZW^T\tilde M$.

We are now ready to describe our two projected eigenvalue bounds. We remark that \eqref{eq:objprojtwo} and
the first inequality in \eqref{eq:Lprojeigbnd}
were already discussed in Proposition~3, Theorem~1 and Theorem~3 in
\cite{ReLiPi:13}. We include them for completeness. We note that the
notation in Lemma~\ref{lem:paramE}, equation \eqref{eq:XZrelate2} and the
next theorem will also be used frequently in Section~\ref{sect:qpbnd} when we discuss the QP lower bound.

\begin{theorem}
\label{thm:projeig}
Let $d\in \Rn$, $G=A-\Diag(d)$.
Let $V$, $W$ be defined in \eqref{eq:defPQ} and $\widehat X=\frac 1n em^T\in
\R^{n\times k}$.
Then:
\begin{enumerate}
\item
\label{item:projbndseqns}
For any $X\in \EE$ and $Z\in \R^{(n-1)\times(k-1)}$ related by
\eqref{eq:XZrelate2}, we have
\begin{equation}
\label{eq:objproj}
\begin{array}{rcl}
\tr GXBX^T &=&
\alpha +
\trace\widehat GZ\widehat BZ^T+
\trace CZ^T
\\ &=&
-\alpha+\tr\widehat GZ\widehat BZ^T+
2\tr G \widehat X BX^T,
\end{array}
\end{equation}
and
\begin{equation}
\begin{array}{rcl}
\label{eq:objprojtwo}
\trace(-L)XBX^T
&=&
\trace\widehat LZ\widehat BZ^T,
\end{array}
\end{equation}
where
\begin{equation}\label{eq:hats}
\widehat{G}=V^TGV, \widehat{L}=V^T(-L)V,~~
\widehat{B}=W^T\tilde{M}B\tilde{M}W,~~
\alpha=\frac 1{n^2}(e^TGe)(m^TBm),~~
C= 2V^TG \widehat X B\tilde{M}W.
\end{equation}
\item
\label{item:2projbnds}
We have the following two lower bounds:
\begin{enumerate}
\item
\label{item:firstAprojbnds}
\beq
\label{eq:Aprojeigbnd}
\begin{array}{rcl}
\cut(m)
&\geq& p^*_{projeig}(G)
:= \displaystyle\frac 12  \left\{
      -\alpha+
\left\langle \lambda(\widehat G),\begin{pmatrix}
         \lambda(\widehat{B})\\0
       \end{pmatrix}\right\rangle_-
+
2\min\limits_{X \in \DD} \trace G\widehat XBX^T
\right\}
\\&=& \displaystyle\frac 12  \left\{
       \alpha+
       \left\langle \lambda(\widehat G),\begin{pmatrix}
         \lambda(\widehat{B})\\0
       \end{pmatrix}\right\rangle_-
+
\min\limits_{0\leq \widehat X + VZW^T\tilde M} \trace CZ^T
\right\}
\\&=& \displaystyle\frac 12  \left\{
      -\alpha+
\sum_{i=1}^{k-2} \lambda_{k-i}(\widehat B)\lambda_i(\widehat G)+
\lambda_{1}(\widehat B)\lambda_{n-1}(\widehat G) +
2\min\limits_{X \in \DD} \trace G\widehat XBX^T
\right\}.
\end{array}
\eeq
\item
\beq
\label{eq:Lprojeigbnd}
\begin{array}{rcl}
\\ \cut(m)
\geq p^*_{projeig}(-L)
&:=&
\displaystyle\frac 12 \left\langle \lambda(\widehat L),\begin{pmatrix}
         \lambda(\widehat{B})\\0
       \end{pmatrix}\right\rangle_-
\geq  p^*_{eig}(-L).
\end{array}
\eeq
\end{enumerate}
\item
The eigenspaces of $V^TLV$
correspond to the eigenspaces of $L$ that are orthogonal to $e$.
\end{enumerate}
\end{theorem}
\begin{proof}
After substituting the parametrization \eqref{eq:XZrelate2} into
the function $\tr GXBX^T$, we obtain a constant, quadratic, and linear term:
\[
\begin{array}{rcl}
\tr GXBX^T
&=&
\trace G (\widehat X + VZW^T\tilde M) B (\widehat X + VZW^T\tilde M)^T
\\&=&
\trace G \widehat X  B \widehat X^T+
\trace(V^TGV)Z(W^T\tilde{M}B\tilde{M}W)Z^T+
\trace 2 V^TG \widehat X B\tilde{M}WZ^T
\end{array}
\]
and
\[
\begin{array}{rcl}
\tr GXBX^T
&=&
\trace G \widehat X  B \widehat X^T+
\trace(V^TGV)Z(W^T\tilde{M}B\tilde{M}W)Z^T+
2 \trace G\widehat X B (VZW^T\tilde M)^T
\\&=&
\trace G \widehat X  B \widehat X^T+
\trace(V^TGV)Z(W^T\tilde{M}B\tilde{M}W)Z^T+
2 \trace G\widehat X B (X-\widehat X)^T
\\&=&
\trace  (-G) \widehat X  B \widehat X^T+
\trace(V^TGV)Z(W^T\tilde{M}B\tilde{M}W)Z^T+
2 \trace G\widehat X B X^T .
\end{array}
\]
These together with \eqref{eq:hats} yield the two equations in \eqref{eq:objproj}.
Since $Le=0$ and hence $L\widehat X=0$, we obtain \eqref{eq:objprojtwo} on replacing $G$ with $-L$ in the above relations.
This proves Item \ref{item:projbndseqns}.

We now prove \eqref{eq:Aprojeigbnd}, i.e.,~Item \ref{item:firstAprojbnds}.
 To this end, recall from \eqref{eq:Prob} and \eqref{eq:partmatrices} that
\[
\cut(m) = \min\left\{\frac12 \tr GXBX^T:\; X\in \DD\cap\DD_O\right\}.
\]
Combining this with \eqref{eq:objproj}, we see further that
\begin{equation}\label{eq:proof1}
  \begin{split}
    \cut(m) &= \frac12\left( -\alpha+\min_{X\in \DD\cap\DD_O}
        \left\{\trace \widehat GZ\widehat BZ^T+
2 \tr G\widehat X B X^T\right\}\right)\\
&\ge \frac12\left( -\alpha+\min_{ X\in \EE\cap\DD_O}\trace
\widehat GZ\widehat BZ^T+
2 \min_{ X\in \DD}\tr G\widehat X B X^T \right)\\
&= \frac 12  \left(
      -\alpha+
\left\langle \lambda(\widehat G),\begin{pmatrix}
         \lambda(\widehat{B})\\0
       \end{pmatrix}\right\rangle_-
+
2\min_{X \in \DD} \trace G\widehat XBX^T
\right)= p^*_{projeig}(G),
  \end{split}
\end{equation}
where $Z$ and $X$ are related via \eqref{eq:XZrelate2}, and the last equality follows from Lemma~\ref{lem:paramE} and
Theorem~\ref{thm:HWineq}. Furthermore, notice that
\begin{equation}\label{eq:XZCrelate}
\begin{split}
&-\alpha  + 2\min_{X \in \DD} \tr G\widehat XBX^T = \alpha + 2\min_{X
\in \DD} \tr G\widehat XB(X - \widehat X)^T\\
=\ & \alpha + 2\min_{0\le \widehat X + VZW^T\tilde M}\tr G\widehat X B (VZW^T\tilde M)^T = \alpha + \min_{0\le \widehat X + VZW^T\tilde M} \tr CZ^T,
\end{split}
\end{equation}
where the second equality follows from Lemma~\ref{lem:paramE}, and the last equality follows from the definition of $C$ in \eqref{eq:hats}.
Combining this last relation with \eqref{eq:proof1} proves the
first two equalities in \eqref{eq:Aprojeigbnd}. The last equality in \eqref{eq:Aprojeigbnd}
follows from the fact that
\begin{equation}\label{eq:eigwidehatB}
\lambda_k(\tilde B)\le \lambda_{k-1}(\widehat B)\le \lambda_{k-1}(\tilde B)\le \cdots\le \lambda_2(\tilde B)= 0 \le \lambda_1(\widehat B)\le \lambda_1(\tilde B),
\end{equation}
which is a consequence of the eigenvalue interlacing theorem
\cite[Corollary~4.3.16]{hj85},
the definition of $\widehat B$ and Lemma~\ref{lem:eigBtilde}.

Next, we prove \eqref{eq:Lprojeigbnd}. Recall again from \eqref{eq:Prob} and \eqref{eq:partmatrices} that
\[
  \cut(m) = \min\left\{\frac12 \tr(-L)XBX^T:\; X\in \DD\cap\DD_O\right\}.
\]
Using \eqref{eq:objprojtwo}, we see further that
\[
\begin{split}
  \cut(m)& \ge \frac12\min\left\{ \tr(-L)XBX^T:\; X\in \EE\cap\DD_O\right\}\\
  & = \frac12\min\left\{ \tr\widehat LZ\widehat BZ^T:\; X\in \EE\cap\DD_O\right\}\\
  & = \frac 12 \left\langle \lambda(\widehat L),\begin{pmatrix}
         \lambda(\widehat{B})\\0
       \end{pmatrix}\right\rangle_- (= p^*_{projeig}(-L))\\
  & \ge \min\left\{\frac12 \tr(-L)XBX^T:\; X\in \DD_O\right\},
\end{split}
\]
where $Z$ and $X$ are related via \eqref{eq:XZrelate2}.
The last inequality follows since the constraint $X\in \EE$ is dropped.

Since $Le = 0$ and the columns of $V$ are orthogonal to $e$, the last conclusion of the theorem follows immediately.
\end{proof}

\begin{remark}
Let $Q \in \R^{(k-1)\times (k-1)}$ be the orthogonal matrix with
columns consisting of the eigenvectors of $\widehat B$, defined in \eqref{eq:hats}, corresponding to
eigenvalues of $\widehat B$ in nondecreasing order;
let $P_G,P_L \in \R^{(n-1)\times (k-1)}$ be the matrices with orthonormal
columns consisting of $k-1$ eigenvectors of $\widehat G, \widehat L$,
respectively, corresponding to the largest $k-2$ in nonincreasing order followed by the smallest.
From \eqref{eq:eigwidehatB} and Theorem~\ref{thm:HWineq},
the minimal scalar product terms in \eqref{eq:Aprojeigbnd} and
\eqref{eq:Lprojeigbnd}, respectively, are attained at
\begin{equation}
\label{eq:ZattainAL}
Z_G=P_G Q^T,\quad  Z_L=P_L Q^T,
\end{equation}
respectively, and two corresponding points in $\EE$ are given, according to \eqref{eq:XZrelate2}, respectively, by
\begin{equation}
\label{eq:XattainAL}
X_G= \widehat X + V Z_GW^T\tilde M, \quad
X_L= \widehat X + V Z_LW^T\tilde M.
\end{equation}
The linear programming problem, \LP, in \eqref{eq:Aprojeigbnd} can be 
solved explicitly; see Lemma \ref{lem:expl} below.
Since the condition number for the symmetric eigenvalue problem is $1$,
e.g.,~\cite{MR98m:65001}, the above shows that we can find the
projected eigenvalue bounds very accurately. In addition, we need only
find $k-1$ eigenvalues of $\widehat G$, $\widehat B$. Hence, if the number of sets
$k$ is small relative to the number of nodes $n$ and the adjacency matrix
$A$ is sparse, then we can find bounds for large problems both
efficiently and accurately; see Section \ref{sect:numertestslarge}.
\end{remark}

\begin{remark}
We emphasize again that although the objective function in \eqref{eq:Prob} is
equivalent for all $d\in \Rn$ on
the set of partition matrices $\Mcal_m$, this is not true once we
relax this feasible set. Though there are advantages to using the
Laplacian matrix as shown in \cite{ReLiPi:13} in terms of simplicity of
the objective function, our numerics suggest that the bound
$p^*_{projeig}(A)$ obtained from
using the adjacency matrix $A$ is stronger than $p^*_{projeig}(-L)$.
Numerical tests confirming this are given  in Section \ref{sect:numertests}.
\end{remark}

\subsubsection{Explicit Solution for Linear Term}
\label{sect:expllin}

The constant term $\alpha$ and eigenvalue minimal scalar product
term of the bound $p^*_{projeig}(G)$ in
\eqref{eq:Aprojeigbnd} can be found efficiently using the two quadratic
forms for $\widehat G$, $\widehat B$ and finding $k-1$ eigenvalues from them.
We now show that
the third term, i.e., the linear term, can also be found efficiently.
Precisely, we give an explicit solution to the linear optimization problem
in \eqref{eq:Aprojeigbnd} in Lemma \ref{lem:expl}, below.

Notice
that in \eqref{eq:Aprojeigbnd}, the minimization is taken over $X\in \DD$, which is shown to be the convex hull of
the set of partition matrices $\Mcal_m$. As mentioned above, this
essentially follows from the Birkhoff and Von Neumann theorems,
see e.g.,~\cite{MR874114}.
Thus, to solve the linear programming problem in \eqref{eq:Aprojeigbnd}, it suffices to consider minimizing the same objective over the nonconvex
set $\Mcal_m$ instead.
\begin{lemma}
\label{lem:expl}
Let $d\in \Rn$, $G=A-\Diag(d)$, $\widehat X=\frac 1n em^T \in \Mcal_m$ and
\[
v_0=
\begin{bmatrix}
(n-m_k-m_1) e_{m_1}
\\
(n-m_k-m_2) e_{m_2}
\\
\vdots
\\
(n-m_k-m_{k-1}) e_{m_{k-1}}
\\
0 e_{m_{k}}
\end{bmatrix},
\]
where $e_j \in \R^j$ is the vector of ones of dimension $j$. Then
\[
\begin{array}{cccc}
 \min\limits_{X\in \Mcal_m} \tr G\widehat XBX^T= \displaystyle\frac 1n\langle Ge,v_0\rangle_-.
\end{array}
\]
\end{lemma}
\begin{proof}
Let $X_0$ denote the feasible partition matrix
\begin{equation}
\label{eq:X0}
X_0 = \begin{bmatrix}
  e_{m_1} & 0 & \cdots & 0\\
  0 & e_{m_2} & \cdots & 0\\
  \vdots & \ddots & \ddots & \vdots\\
  0 & \cdots & 0 & e_{m_k}
\end{bmatrix} \in \Mcal_m.
\end{equation}
Then it is clear that $X\in \Mcal_m$ if, and only if, there exists a permutation matrix $P$ on $\{1,\ldots,n\}$ so that
$X = PX_0$. Using this observation and letting $S_N$ denote the set of permutation matrices on $\{1,\ldots,n\}$, we have
\begin{equation*}
  \begin{split}
   \min_{X\in \Mcal_m} \tr G\widehat XBX^T & = \frac{1}{n}\min_{P\in S_N}\tr Ge m^TBX_0^TP^T \\
   & = \frac{1}{n}\min_{P\in S_N}\tr Ge (X_0 Bm)^TP^T\\
   & = \frac{1}{n}\min_{P\in S_N}\tr Ge v_0^TP^T = \frac{1}{n}\langle Ge,v_0\rangle_-,
  \end{split}
\end{equation*}
where the last equality follows from the definition of minimal scalar product.
\end{proof}


\section{Quadratic Programming Lower Bound}
\label{sect:qpbnd}

A new successful and efficient bound used for the
QAP is given in \cite{AnsBrix:99,BrixiusAnstr:01}.
In this section, we adapt the idea described there to obtain a lower bound
for ${\rm cut}(m)$. This bound uses a relaxation that is a {\em convex} QP,
i.e.,~the minimization of a quadratic function that is convex on
the feasible set defined by linear inequality constraints. Approaches based
on nonconvex QPs are given in e.g.,~\cite{HagerHung:13} and the references
therein.
\index{QP, quadratic program}
\index{quadratic program, QP}
\index{QAP, quadratic assignment problem}
\index{quadratic assignment problem, QAP}

The main idea in \cite{AnsBrix:99,BrixiusAnstr:01} is to use the
zero duality gap result for a homogeneous QAP
\cite[Theorem~3.2]{AnWo:98} on an objective obtained via a suitable reparametrization of the original problem.
Following this idea, we consider the parametrization in \eqref{eq:objproj}
where our main objective in \eqref{eq:Prob} is rewritten as:
\begin{equation}\label{relation}
\frac 12 \tr GXBX^T = \frac 12 \left(\alpha +
\trace\widehat GZ\widehat BZ^T+
\trace CZ^T\right)
\end{equation}
with $X$ and $Z$ related according to \eqref{eq:XZrelate}, and $G = A - \Diag(d)$ for some $d\in \Rn$.
We next look at the homogeneous part:
\begin{equation}\label{eq:1strelaxation}
  \begin{array}{rl}
    v_r^*:= \min & \frac 12\trace\widehat G Z\widehat B Z^T\\
    {\rm s.t.} & Z^TZ = I.
  \end{array}
\end{equation}
Notice that the constraint $ZZ^T\preceq I$ is redundant for the above problem. By adding this redundant constraint, the corresponding Lagrange dual problem is given by
\begin{equation}\label{eq:1strelaxationdual}
  \begin{array}{rl}
    v_{dsdp}:= \max  & \frac{1}{2}\trace S + \frac{1}{2}\trace T\\
    \text{s.t.} &
      I_{k-1} \otimes S + T\otimes I_{n-1} \preceq
      \widehat B\otimes \widehat G, \\
         &   S\preceq 0, \\
         & S \in \S^{n-1},\ T\in \S^{k-1},
  \end{array}
\end{equation}
where the variables $S$ and $T$ are the dual variables corresponding to the constraints $ZZ^T \preceq I$ and $Z^TZ = I$, respectively.
It is known that $v^*_r= v_{dsdp}$; see \cite[Theorem~2]{PoRe:08}.
This latter problem \eqref{eq:1strelaxationdual} can be solved efficiently.
For example, as in the proofs of \cite[Theorem~3.2]{AnWo:98}
and \cite[Theorem~2]{PoRe:08},
one can take advantage of the properties of the Kronecker product
and orthogonal diagonalizations of $\widehat B, \widehat G$,
to reduce the problem to solving the following \LP
with $n+k-2$ variables,
\begin{equation}\label{eq:LP}
  \begin{array}{rl}
    \max  & \frac{1}{2}e^Ts + \frac{1}{2}e^Tt\\
    \text{s.t.} &
      t_i + s_j \le \lambda_i\sigma_j,\ \ i = 1,\ldots, k-1,\ \ j = 1,\ldots,n-1,\\
      & s_j \le 0,\ \ j = 1,\ldots,n-1,
  \end{array}
\end{equation}
where
\begin{equation}\label{eq:eigAB}
\widehat B = U_1\Diag(\lambda) U_1^T\ \ {\rm and}\ \ \widehat G = U_2\Diag(\sigma) U_2^T
\end{equation}
are eigenvalue orthogonal decompositions of $\widehat B$ and $\widehat G$,
respectively. From an optimal solution $(s^*,t^*)$ of \eqref{eq:LP},
we can recover an optimal solution of \eqref{eq:1strelaxationdual} as
\begin{equation}\label{eq:ST}
S^* = U_2\Diag(s^*)U_2^T\ \ \ \ \ T^* = U_1\Diag(t^*)U_1^T.
\end{equation}

Next, suppose that the optimal value of the dual problem \eqref{eq:1strelaxationdual} is attained at $(S^*,T^*)$. Let $Z$ be such that the $X$ defined according to \eqref{eq:XZrelate}
is a partition matrix. Then we have
\begin{equation*}
\begin{split}
    \frac{1}{2}\trace(\widehat G Z\widehat B Z^T)
    &
= \frac{1}{2}\kvec(Z)^T(\widehat B\otimes \widehat G)
\kvec(Z)\\
    &= \frac{1}{2}\kvec(Z)^T\underbrace{(\widehat B\otimes \widehat G - I \otimes S^* - T^*\otimes I)}_{\widehat Q} \kvec(Z) + \frac{1}{2}\trace(ZZ^TS^*) + \frac{1}{2}\trace(T^*)\\
    & = \frac{1}{2}\kvec(Z)^T\widehat Q\kvec(Z) + \frac{1}{2}\trace([ZZ^T - I]S^*) + \frac{1}{2}\trace(S^*) + \frac{1}{2}\trace(T^*)\\
    & \ge \frac{1}{2}\kvec(Z)^T\widehat Q\kvec(Z) + \frac{1}{2}\trace(S^*) + \frac{1}{2}\trace(T^*),
\end{split}
\end{equation*}
where the last inequality uses $S^*\preceq 0$ and $ZZ^T\preceq I$.

Recall that the original nonconvex problem \eqref{eq:Prob} is
equivalent to minimizing the right hand side of \eqref{relation} over the set of all $Z$ so that the
$X$ defined in \eqref{eq:XZrelate} corresponds to a partition matrix. From the above relations, the third equality in \eqref{eq:partmatrices}
and Lemma~\ref{lem:paramE}, we see that
\begin{equation}\label{eq:formulation}
\begin{array}{rc}
  \cut(m)\ge \min & \frac 12(\alpha +\trace C Z^T + \kvec(Z)^T\widehat Q \kvec(Z)) + \frac{1}{2}\trace(S^*) + \frac{1}{2}\trace(T^*)\\
  [3 pt]
  {\rm s.t.}&  Z^TZ = I_{k-1},\ \ VZW^T\tilde M\ge -\widehat X.
\end{array}
\end{equation}
We also recall from \eqref{eq:1strelaxationdual} that $\frac{1}{2}\trace(S^*) + \frac{1}{2}\trace(T^*)= v_{dsdp} = v_r^*$, which
further equals
\begin{equation*}
  \frac 12\left\langle \lambda(\widehat G),\begin{pmatrix}
         \lambda(\widehat{B})\\0
       \end{pmatrix}\right\rangle_-
\end{equation*}
according to \eqref{eq:1strelaxation} and Theorem~\ref{thm:HWineq}.

A lower bound can now be obtained by relaxing the constraints in \eqref{eq:formulation}.
For example, by dropping the orthogonality constraints, we obtain the following lower bound on $\cut(m)$:
\begin{equation}\label{eq:qplb2}
  \begin{array}{rl}
    p^*_{QP}(G) := \min  & q_1(Z) := \frac 12\left(\alpha +\trace C Z^T +
\kvec(Z)^T\widehat Q \kvec(Z) + \left\langle \lambda(\widehat G),\begin{pmatrix}
         \lambda(\widehat{B})\\0
       \end{pmatrix}\right\rangle_-\right)\\
    {\rm s.t.} & V ZW^T\tilde M\ge -\widehat X,
  \end{array}
\end{equation}
Notice that this is a QP with $(n-1)(k-1)$ variables and $nk$ constraints.

As in \cite[Page~346]{AnsBrix:99}, it is possible to reformulate \eqref{eq:qplb2} into a QP in variables $X\in \DD$.
Note that $\tilde Q$ defined in \eqref{eq:tildeQ0} is not positive semidefinite in general. Nevertheless, the QP is implicitly convex.
\begin{theorem}
Let $S^*,T^*$ be optimal solutions of \eqref{eq:1strelaxationdual} as
defined in \eqref{eq:ST}.
A lower bound on $\cut(m)$ is obtained from the following QP:
  \begin{equation}\label{eq:qplb3}
  \cut(m) \geq p^*_{QP}(G) = \min_{X\in \DD}\frac{1}{2}\kvec(X)^T\tilde Q\kvec(X) +
  \frac{1}{2}\left\langle \lambda(\widehat G),\begin{pmatrix}
         \lambda(\widehat{B})\\0
       \end{pmatrix}\right\rangle_-
  \end{equation}
  where
  \begin{equation}\label{eq:tildeQ0}
  \tilde Q := B\otimes G - M^{-1}\otimes VS^*V^T - {\tilde M}^{-1}WT^*W^T{\tilde M}^{-1}\otimes I_n.
  \end{equation}
  The QP in \eqref{eq:qplb3} is implicitly convex since $\tilde Q$ is positive semidefinite on the tangent space of $\EE$.
\end{theorem}
\begin{proof}
  We start by rewriting the second-order term of $q_1$ in \eqref{eq:qplb2} using the relation \eqref{eq:XZrelate}.
  Since $V^TV = I_{n-1}$ and $W^TW=I_{k-1}$, we have from the definitions of $\widehat B$ and $\widehat G$ that
  \begin{equation}\label{eq:hatQtildeQ}
  \begin{split}
    \widehat Q &= \widehat B\otimes \widehat G - I_{k-1}\otimes S^* - T^*\otimes I_{n-1}\\
    &= W^T\tilde M B \tilde M W \otimes V^TGV - I_{k-1}\otimes S^* - T^*\otimes I_{n-1}\\
    &= (\tilde M W\otimes V)^T[B\otimes G - M^{-1}\otimes VS^*V^T - {\tilde M}^{-1}WT^*W^T{\tilde M}^{-1}\otimes I_n](\tilde M W\otimes V)
  \end{split}
  \end{equation}
  On the other hand, from \eqref{eq:XZrelate2}, we have
  \[
  \kvec(X - \widehat X) = \kvec(VZW^T\tilde M) = (\tilde M W\otimes V)\kvec(Z).
  \]
  Hence, the second-order term in $q_1$ can be rewritten as
  \begin{equation}\label{eq:Q0rewrite}
    \kvec(Z)^T\widehat Q\kvec(Z)
      = \kvec(X - \widehat X)^T\tilde Q\kvec(X - \widehat X),
  \end{equation}
  where $\tilde Q$ is defined in \eqref{eq:tildeQ0}.
  Next, we see from $V^T e = 0$ that
  \[
    (M^{-1}\otimes VS^*V^T)\kvec(\widehat X) = \frac{1}{n}(M^{-1}\otimes VS^*V^T)(m\otimes I_{n})e = \frac{1}{n}(e\otimes VS^*V^T)e = 0.
  \]
  Similarly, since $W^T\tilde m = 0$, we also have
  \[
  \begin{split}
    ({\tilde M}^{-1}WT^*W^T{\tilde M}^{-1}\otimes I_n)\kvec(\widehat X)
    &= \frac{1}{n}({\tilde M}^{-1}WT^*W^T{\tilde M}^{-1}\otimes I_n)(m\otimes I_{n})e\\
    & = \frac{1}{n}({\tilde M}^{-1}WT^*W^T\tilde m\otimes I_n)e = 0.
  \end{split}
  \]
  Combining the above two relations with \eqref{eq:Q0rewrite}, we obtain further that
  \[
  \begin{split}
    &\kvec(Z)^T\widehat Q\kvec(Z)\\
     = &\kvec(X)^T\tilde Q \kvec(X) - 2 \kvec(\widehat X)^T[B\otimes G]\kvec(X) + \kvec(\widehat X)[B\otimes G]\kvec(\widehat X)\\
     = &\kvec(X)^T\tilde Q \kvec(X) - 2\trace G\widehat X B X^T + \alpha.
  \end{split}
  \]
  For the first two terms of $q_1$, proceeding as in \eqref{eq:XZCrelate}, we have
  \[
  \alpha + \trace CZ^T = -\alpha + 2 \trace G\widehat X B X^T.
  \]
  Furthermore, recall from Lemma~\ref{lem:paramE} that with $X$ and $Z$ related by \eqref{eq:XZrelate},
  $X\in \DD$ if, and only if, $VZW^T\tilde M \ge - \widehat X$.

  The conclusion in \eqref{eq:qplb3} now follows by substituting the above expressions into \eqref{eq:qplb2}.

Finally, from \eqref{eq:hatQtildeQ} we see that $\tilde Q$ is positive semidefinite
when restricted to the range of $\tilde M W\otimes V$. This is precisely the tangent space of $\EE$.
\end{proof}

Although the dimension of the feasible set in \eqref{eq:qplb3} is slightly
larger than the dimension of the feasible set in \eqref{eq:qplb2},
the former feasible set is much simpler. Moreover, as mentioned above,
even though $\tilde Q$ is not positive semidefinite in general,
it is when restricted to the tangent space of $\EE$.
Thus, as in \cite{BrixiusAnstr:01}, one may apply the Frank-Wolfe algorithm
on \eqref{eq:qplb3} to approximately compute the QP lower bound $p^*_{QP}(G)$
for problems with huge dimension.

Since $\widehat Q\succeq 0$, it is easy to see from \eqref{eq:qplb2} that $p^*_{QP}(G)\ge p^*_{projeig}(G)$.
This inequality is not necessarily strict. Indeed, if $G = -L$, then $C=0$ and $\alpha = 0$ in \eqref{eq:qplb2}. Since the feasible set
of \eqref{eq:qplb2} contains the origin, it follows from this and the definition of $p^*_{projeig}(-L)$ that $p^*_{QP}(-L)= p^*_{projeig}(-L)$.
Despite this, as we see in the numerics Section \ref{sect:numertests},
we have $p^*_{QP}(A)> p^*_{projeig}(A)$ for most of our numerical experiments. In general, we still do not know what conditions will guarantee $p^*_{QP}(G)> p^*_{projeig}(G)$.

\section{Semidefinite Programming Lower Bounds}
\label{sect:sdpbnd}

In this section, we study the SDP relaxation constructed from the various equality constraints in the representation in
\eqref{eq:partmatrices} and the objective function in \eqref{eq:objfn}.

One way to derive an SDP relaxation for \eqref{eq:Prob} is to start by considering a suitable Lagrangian relaxation, which is itself an SDP.
Taking the dual of this Lagrangian relaxation then gives an SDP relaxation
for \eqref{eq:Prob}; see~\cite{KaReWoZh:94} and~\cite{WoZh:96} for the
development for the QAP and GP cases, respectively.
Alternatively, we can also obtain the {\em same} SDP relaxation directly
using the well-known \emph{lifting process},
e.g.,~\cite{BaCeCo:93,LoSc:91,ShAd:96,WoZh:96,KaReWoZh:94}.
In this approach, we start with the following equivalent quadratically
constrained quadratic problems
to \eqref{eq:Prob}:
\index{$A \circ B$, Hadamard product}
\index{Hadamard product, $A \circ B$}
\index{$G = A -\Diag(d), d\in \Rn$}
\beq
\label{eq:qqp1sq}
\begin{array}{rcllll}
\cut(m)
  =&\min& \frac{1}{2} \tr GXBX^T                  =&\min& \frac{1}{2} \tr GXBX^T \\
&\mbox{s.t.} & X \circ X = X,                      &\mbox{s.t.} & X \circ X = x_0X,\\
&&\|Xe-e\|^{2}=0,                                  &&\|Xe-x_0e\|^{2}=0, \\
&&\|X^Te-m\|^{2}=0,                                &&\|X^Te-x_0m\|^{2}=0, \\
&& X_{:i} \circ X_{:j} =0, ~\forall i \neq j,      && X_{:i} \circ X_{:j} =0, ~\forall i \neq j,\\
&& X^TX-M =0,                                      && X^TX-M =0,\\
&& \diag(XX^T)-e =0.                               && \diag(XX^T) - e =0,\\
&&                                                 && x_0^2 = 1.
\end{array}
\end{equation}
Here: $G = A -\Diag(d), ~ d\in \Rn$;
the first equality follows from the fifth equality in \eqref{eq:partmatrices},
and we add $x_0$ and the constraint $x_0^2=1$ to \emph{homogenize} the problem.
Note that if $x_0=-1$ at the optimum, then we can replace it with
$x_0=1$ by changing the sign $X\leftarrow -X$ while leaving the
objective value unchanged.
We next linearize the quadratic terms in \eqref{eq:qqp1sq} using the matrix
\[
Y_{X}:=\left (
\begin{array}{c}
  1\\
\kvec(X)
\end{array}
\right )
(1~~\kvec(X)^T).
\]
Then $Y_X \succeq 0$ and is rank one.
The objective function becomes
\index{$L_G$, objective}
\[
  \frac 12\tr GXBX^T = \frac 12\tr L_G Y_X,
\]
where
\beq
\label{eq:la1}
L_G := \left[ \begin{array}{cc}
0 & 0 \\ 0 & B \otimes G
\end{array} \right].
\end{equation}
By removing the rank one restriction on $Y_X$
and using a general symmetric matrix variable $Y$ rather than $Y_X$,
we obtain the following SDP relaxation and its properties:
\beq \label{eq:sdp1}
\begin{array}{rl}
  \cut(m) \geq p_{SDP}^*(G):=\min & \frac12\tr L_{G} Y \\
  \mbox{~s.t.~}
       & \arrow(Y)   = e_0,   \\
       & \tr D_{1} Y  = 0, \\
       & \tr D_{2} Y  = 0, \\
       & \GG_{J}(Y) = 0, \\
       & \DD_O(Y) = M, \\
       & \DD_e(Y) = e, \\
       & Y_{00} = 1, \\
       & Y \succeq 0,
\end{array}
\end{equation}
where the rows and columns of $Y\in \S^{kn+1}$ are indexed from $0$ to
$kn$. We now describe the constraints in detail.
\begin{enumerate}
\item
The {\em arrow linear transformation} acts on $\S^{kn+1}$,
\index{$\arrow$, arrow constraint}
\index{arrow constraint, $\arrow$}
\beq \label{eq:arrow}
\arrow (Y) := \diag (Y) - (0,Y_{0,1:kn})^T,
\end{equation}
$Y_{0,1:kn}$ is the vector formed from the last $kn$ components
of the first row (indexed by $0$) of $Y$.
The arrow constraint represents $X \in \ZZ$, and
$e_0$ is the first ($0$th) unit vector.
\index{$\GG_J$, gangster constraint}
\index{gangster constraint, $\GG_J$}
\item
The norm constraints for $X\in \EE$ are represented
by the constraints with the two $(kn+1) \times (kn+1)$ matrices
\[
D_1 :=
  \left[  \begin{array}{cc}
 n & -e_{k}^T \otimes e_{n}^T \\
 -e_{k} \otimes e_{n} & (e_{k}e_{k}^T) \otimes I_{n}
\end{array}  \right],
\]
\[
D_2 :=
  \left[  \begin{array}{cc}
 m^Tm & -m^T  \otimes e_{n}^T \\
 -m \otimes e_{n} & I_{k} \otimes (e_{n}e_{n}^T)
\end{array}  \right].
\]
\item
We let $\GG_J$ represent the gangster operator on ${\S}^{kn+1}$, i.e.,~it
shoots \emph{holes} in a matrix,
\beq
\label{eq:gangster}
({\mathcal G}_{J}(Y))_{ij}:=\left\{\begin{array}{ll}
                      Y_{ij} & \mbox{if}~ (i,j)~ \mbox{or}~(j,i)~\in J\\
                      0     & \mbox{otherwise,}
                    \end{array}
                    \right.
\end{equation}
\[
J:= \left \{(i,j): i=(p-1)n+q,~~j=(r-1)n+q,~~\mbox{for}~~\begin{array}{l}
              p<r,~ p,r \in\{1,\ldots,k\}\\
             q \in \{1,\ldots,n\}
                \end{array}  \right \}.
\]
The gangster constraint represents the (Hadamard) orthogonality of the
columns. The zeros are the diagonal elements of the off-diagonal blocks
$\bar Y_{(ij)}, 1<i<j,$ of
$Y$; see the block structure in \eqref{eq:blocked} below.
\item
Again, by abuse of notation, we use the symbols for the sets of constraints
$\DD_O,\DD_e$ to represent the linear transformations in the SDP
relaxation \eqref{eq:sdp1}. Note that
\[
   \begin{array}{rcl}
  \langle \Psi, X^TX \rangle = \trace I X \Psi X^T
                    = \kvec(X)^T (\Psi  \otimes I) \kvec(X).
\end{array}
\]
Therefore, the adjoint of $\DD_O$ is made up of a zero
row/column and $k^2$ blocks that are multiples of the identity:
\[
\DD_O^*(\Psi)=
\begin{bmatrix}
0 | & 0 \cr
\hline
0 |  &  \Psi  \otimes I_n
\end{bmatrix}.
\]
If $Y$ is blocked appropriately as
\beq
\label{eq:blocked}
Y=
\begin{bmatrix}
Y_{00} | & Y_{0,:} \cr
\hline
Y_{:,0} |  &  \bar Y
\end{bmatrix}, \quad
\bar Y = \begin{bmatrix}
  \bar Y_{(11)} & \bar Y_{(12)} & \cdots & \bar Y_{(1k)}\\
  \bar Y_{(21)} & \bar Y_{(22)} & \cdots & \bar Y_{(2k)}\\
  \vdots & \ddots & \ddots & \vdots\\
  \bar Y_{(k1)} & \ddots& \ddots & \bar Y_{(kk)}
\end{bmatrix},
\eeq
with each $\bar Y_{(ij)}$ being a $n\times n$ matrix,
then
\beq
\label{eq:constrDDO}
\DD_O(Y)= \left(\trace \bar Y_{(ij)}\right) \in \Sk.
\eeq

Similarly,
\[
  \langle \phi, \diag(XX^T) \rangle
  =\langle \Diag(\phi), XX^T \rangle
  =\kvec(X)^T\left(I_k \otimes \Diag(\phi)\right) \kvec(X).
\]
Therefore we get the sum of the diagonal parts
\beq
\label{eq:constrDDe}
\DD_e(Y)= \sum_{i=1}^k \diag \bar Y_{(ii)} \in \R^n.
\eeq
\end{enumerate}

\subsection{Final SDP Relaxation}

We present our final SDP relaxation \eqref{eq:SDPfinal}
in Theorem~\ref{thm:finalSDP} below and discuss some of its properties.
This relaxation is surprisingly simple/strong with many of the constraints
in \eqref{eq:sdp1} redundant. In particular, we show that
the problem is independent of the choice of $d\in \Rn$ in constructing $G$.
We also show that the two
constraints using $\DD_O,\DD_e$ are redundant in the SDP relaxation \eqref{eq:SDPfinal}. This answers affirmatively the question posed in \cite{WoZh:96}
on whether these constraints were redundant in the SDP relaxation for the GP.

Since both $D_1$ and $D_2$ are positive semidefinite and $\trace D_i Y=0,
i=1,2$, we conclude that the feasible set of \eqref{eq:sdp1}
has no strictly feasible (positive definite) points, $Y\succ 0$.
Numerical difficulties can arise when an
interior-point method is directly applied to a problem where strict feasibility,
Slater's condition, fails. Nonetheless, we can find a very simple
structured matrix in the relative interior of the feasible set to project (and regularize) the problem into a smaller dimension.
As in \cite{WoZh:96}, we achieve this by finding a matrix $V$ with range equal to the intersection
of the nullspaces of $D_1$ and $D_2$. This is called
\emph{facial reduction}, \cite{bw1,ScTuWonumeric:07}.
\index{facial reduction}
Let $ V_j \in \R^{j\times (j-1)}$, $ V_j^T e = 0$, e.g.,
\[
V_j:=\left[
\begin{array}{ccccc}
1 & 0& \ldots &\ldots &0\\
0 & 1&\ldots &\ldots &0\\
0 & 0 &1&\ldots&0 \\
\ldots & \ldots &\ldots& \ldots& 1\\
-1  & \ldots&\ldots & -1 &-1
\end{array}
\right]_{j \times (j-1)}.
\]
and let
\[
\widehat{V}:=
\left[
\begin{array}{cc}
1 & 0 \\
\frac{1}{n}{m} \otimes e_{n} & V_{k} \otimes V_{n}
\end{array}
\right].
\]
Then the range of $\widehat V$ is equal to the range of (any)
$ \widehat Y \in \relint F$, the relative interior of the minimal face.
And, we can facially reduce \eqref{eq:sdp1} using the substitution
\[
 Y= \widehat V Z \widehat V^T \in \S^{kn+1},  \quad Z \in \S^{(k-1)(n-1)+1}.
\]

The facially reduced SDP is then
\beq
\label{eq:facred1}
\begin{array}{rccl}
  \cut(m) \geq p_{SDP}^*(G)&=&
  \min & \frac12\tr \widehat{V}^{T} L_{G} \widehat{V}Z \\
  &&\text{s.t.}
       & \arrow(\widehat{V}Z\widehat{V}^{T})   = e_0   \\
       &&& \GG_{J}(\widehat{V}Z\widehat{V}^{T}) = 0 \\
       &&& (\widehat{V}Z\widehat{V}^{T})_{00} = 1 \\
       &&& \DD_O(\widehat{V}Z\widehat{V}^{T}) = M \\
       &&& \DD_e(\widehat{V}Z\widehat{V}^{T}) = e \\
       &&& Z \succeq 0, ~~ Z \in \S^{(k-1)(n-1)+1}.
\end{array}
\eeq
We let $ \bar J := J \cup {(0,0)}$.
\index{$\bar J := J \cup {(0,0)}$}
Our main, simplified, SDP relaxation is as follows.
\begin{theorem}
\label{thm:finalSDP}
The facially reduced SDP \eqref{eq:facred1} is equivalent to the single equality constrained problem
\begin{equation}\label{eq:SDPfinal}\tag{{SDP{$_{final}$}}}
\begin{array}{rccl}
  \cut(m) \geq p_{SDP}^*(G)&=&
  \min & \frac12\tr \left(\widehat{V}^{T} L_{G} \widehat{V}\right)Z \\
  && {\rm s.t.}
       & {\GG}_{\bar J}(\widehat{V}Z\widehat{V}^{T}) = \GG_{\bar J} (e_0e_0^T) \\
       &&& Z \succeq 0, ~~ Z \in \S^{(k-1)(n-1)+1}.
\end{array}
\end{equation}
The dual program is
\beq
\label{eq:SDPdualfinal}
\begin{array}{ll}
  \max & \frac12 W_{00} \\
  {\rm s.t.}
       & \widehat V^T \GG_{\bar J}(W) \widehat V \preceq
                  \widehat V^T L_G \widehat V
\end{array}
\eeq
Both primal and dual satisfy Slater's constraint qualification and the objective function is
independent of the $d\in \Rn$ chosen to form $G$.
\end{theorem}
\begin{proof}
It is shown in \cite{WoZh:96} that the second and third constraint in
\eqref{eq:facred1} along with $Z\succeq 0$ implies that the $\arrow$
constraint holds, i.e. the $\arrow$ constraint is redundant.
It only remains to show that the last two equality constraints in
\eqref{eq:facred1} are redundant. First, the gangster constraint implies that
the blocks in $Y=\widehat{V}Z\widehat{V}^{T}$ satisfy
$\diag\bar Y_{(ij)} = 0$ for all $i\neq j$.
Next, notice that $D_i\succeq 0$, $i=1$, $2$. Moreover, using $Y\succeq 0$ and considering the Schur complement of $Y_{00}$, we have
\[
Y \succeq Y_{0:kn,0} Y_{0:kn,0}^T.
\]
Writing $v_1 := Y_{0:kn,0}$ and $X = {\rm Mat}(Y_{1:kn,0})$, we see further that
\[
0 = \trace(D_iY) \ge \trace(D_iv_1v_1^T) = \begin{cases}
  \|Xe - e\|^2 & {\rm if}\ i = 1,\\
  \|X^Te - m\|^2 & {\rm if}\ i = 2.
\end{cases}
\]
This together with the arrow constraints show that $\trace\bar Y_{(ii)} = \sum_{j=(i-1)n+1}^{ni}Y_{j0} = m_i$. Thus,
$\DD_O(\widehat{V}Z\widehat{V}^{T}) = M$ holds.
Similarly, one can see from the above and the arrow constraint that
$\DD_e(\widehat{V}Z\widehat{V}^{T}) = e$ holds.

The conclusion about Slater's constraint qualification for
\eqref{eq:SDPfinal} follows from \cite[Theorems~4.1]{WoZh:96},
which discussed the primal SDP relaxations of the GP. That relaxation
has the same feasible set as \eqref{eq:SDPfinal}.
In fact, it is shown in \cite{WoZh:96}
that
\[
\hat{Z} = \left[
\begin{array}{c|c}
1 & 0 \\
\hline\\
0  &
\frac{1}{n^{2}(n-1)}(n\Diag(\bar{m}_{k-1})-\bar{m}_{k-1}\bar{m}_{k-1}^T)
\otimes (nI_{n-1}-E_{n-1})
\end{array}
\right]
       \in \S_+^{(k-1)(n-1)+1},
\]
where $\bar{m}_{k-1}^T =(m_{1},\ldots, m_{k-1})$ and $E_{n-1}$ is the $n-1$ square matrix of ones,
is a strictly feasible point for \eqref{eq:SDPfinal}.
The right-hand side of the dual \eqref{eq:SDPdualfinal} differs from
the dual of the SDP relaxation of the GP.
However, let
\[
\hat{W}=  \left[
\begin{array}{cc}
\alpha & 0 \\
0  & (E_{k}-I_{k}) \otimes I_{n}
\end{array}
\right].
\]
From the proof of \cite[Theorems~4.2]{WoZh:96} we see that
$\GG_{\bar J}(\hat W) =\hat W$ and
\[ \begin{array}{ll}
-\widehat{V}^T \GG_{\bar J}(\hat{W})
\widehat{V}
&=
\widehat{V}^T (-\hat{W})
\widehat{V}
\\&=
      \left [ \begin{array}{cc}
                                        1 & m^T \otimes e^T/n \\
                                        0 & V_{k}^T \otimes V_{n}^T
                                              \end{array}
                                      \right ]
                   \left [ \begin{array}{cc}
                                        -\alpha & 0 \\
                                        0 & ((I_{k}-E_{k}) \otimes I_{n}
                                              \end{array}
                                      \right ]
                   \left [ \begin{array}{cc}
                                        1 & 0 \\
                   m \otimes e/n & V_{k} \otimes V_{n}
                                              \end{array}
                                      \right ] \\
     &= \left [ \begin{array}{cc}
      -\alpha + m^T(I_{k}-E_{k})m/n
      & (m^T(I_{k}-E_{k}) V_{k})  \otimes (e^T  V_{n})/n \\
      (V_{k}^T(I_{k}-E_{k}) m) \otimes (V_{n}^T e)/n
      & (V_{k}^T(I_{k}-E_{k}) V_{k}) \otimes (V_{n}^T  V_{n})
                \end{array}
        \right ] \\
     &= \left [ \begin{array}{cc}
      -\alpha+m^T(I_{k}-E_{k})m/n
      & 0  \\
      0
      & (I_{k-1}+E_{k-1}) \otimes (I_{n-1}+E_{n-1})
                \end{array}
        \right ]
  \\&  \succ 0,  \qquad \text{ for sufficiently large } -\alpha.
    \end{array}
\]
Therefore
${\widehat V}^T \GG_{\bar J}(\beta \hat W)\widehat V \prec \widehat V^T L_G \widehat V$
for sufficiently large $-\alpha, \beta$, i.e.,~Slater's
constraint qualification holds for the dual \eqref{eq:SDPdualfinal}.

Finally, we let $Y= \widehat{V}Z\widehat{V}^{T}$ with $Z$ feasible for
\eqref{eq:SDPfinal}. Then $Y$ satisfies the gangster constraints, i.e., $\diag\bar Y_{(ij)}  = 0$ for all $i\neq j$.
On the other hand, if we restrict $D=\Diag(d)$, then the objective matrix $L_D$ has
nonzero elements only in the same diagonal positions of the off-diagonal
blocks from the application of the Kronecker product $B \otimes
\Diag(d)$. Thus, we must have $\trace L_DY = 0$. Consequently, for all $d\in \Rn$,
\[
\trace \left(\widehat{V}^TL_G\widehat{V}\right)Z= \trace L_G\widehat{V}Z\widehat{V}^{T}=\trace L_G Y =
\trace L_A Y =
\trace \widehat{V}L_A\widehat{V}^{T}Z.
\]
\end{proof}
The above Theorem \ref{thm:finalSDP} also answers a question posed in
\cite{WoZh:96}, i.e.,~whether the two
constraints $\DD_O,\DD_e$ are redundant in the corresponding SDP relaxation of GP.
Surprisingly, the answer is yes, they are both redundant.

We next present
two useful properties for finding/recovering approximate solutions $X$ from a solution $Y$ of \eqref{eq:SDPfinal}.
\begin{prop}
\label{lem:linequals}
Suppose that $Y$ is feasible for \eqref{eq:SDPfinal}. Let $v_1=Y_{1:kn,0}$
and $\begin{pmatrix}
  v_0 & v_2^T
\end{pmatrix}^T$ denote a unit eigenvector of $Y$ corresponding to the largest eigenvalue. Then
$X_1:=\Mat(v_1)\in \EE\cap \NN$. Moreover, if $v_0\neq 0$, then
$X_2:=\Mat(\frac 1{v_0}v_2) \in \EE$. Furthermore, if, $Y \geq 0$, then
$v_0 \neq 0$ and $X_2\in \NN$.
\end{prop}
\begin{proof}
The fact that $X_1\in \EE$ was shown in the proof of Theorem~\ref{thm:finalSDP}. That $X_1 \in \NN$ follows from the arrow constraint.
We now prove the results for $X_2$. Suppose first that $v_0\neq 0$. Then
\[
Y \succeq \lambda_1(Y)\begin{pmatrix}
  v_0\\ v_2
\end{pmatrix}\begin{pmatrix}
  v_0\\ v_2
\end{pmatrix}^T.
\]
Using this and the definitions of $D_i$ and $X_2$, we see further that
\begin{equation}\label{eq:XinEE}
0 = \trace(D_iY) \ge \begin{cases}
  \lambda_1(Y)v_0^2\|X_2e - e\|^2, & {\rm if}\ i = 1,\\
  \lambda_1(Y)v_0^2\|X_2^Te - m\|^2, & {\rm if}\ i = 2.
\end{cases}
\end{equation}
Since $\lambda_1(Y)\neq 0$ and $v_0\neq 0$, it follows that $X_2\in \EE$.

Finally, suppose that $Y\ge 0$. We claim that any eigenvector $\begin{pmatrix}
  v_0 & v_2^T
\end{pmatrix}^T$ corresponding to the largest eigenvalue must satisfy:
\begin{enumerate}
  \item $v_0\neq 0$;
  \item all entries have the same sign, i.e., $v_0v_2\ge 0$.
\end{enumerate}
From these claims, it would follow immediately that $X_2=\Mat(v_2/v_0)\in \NN$.

To prove these claims, we note first from the classical
Perron-Fr\"{o}benius theory, e.g.,~\cite{BrualRys:91}, that the vector $\begin{pmatrix}
  |v_0| & |v_2|^T
\end{pmatrix}^T$ is also an eigenvector corresponding to the largest eigenvalue.\footnote{Indeed, if $Y$ is irreducible,
the top eigenspace must be the span of a positive vector. Hence the conclusion follows. For a reducible $Y$,
the top eigenspace must then be a direct product of the top eigenspaces of each irreducible block. The conclusion follows similarly.}
Letting $\chi := \Mat(v_2)$ and proceeding as in \eqref{eq:XinEE}, we conclude that
\[
\|\chi e - v_0 e\|^2 = 0 \ \ {\rm and} \ \ \||\chi|e - |v_0|e\|^2 = 0.
\]
The second equality implies that $v_0\neq 0$. If $v_0>0$, then for all $i=1,\cdots,n$,
we have
\[
\sum_{j=1}^k \chi_{ij} = v_0 = \sum_{j=1}^k|\chi_{ij}|,
\]
showing that $\chi_{ij}\ge 0$ for all $i$, $j$, i.e., $v_2\ge 0$. If $v_0<0$, one can show similarly that $v_2\le 0$. Hence, we have also shown $v_0v_2\ge 0$.
This completes the proof.
\end{proof}


\section{Feasible Solutions and Upper Bounds}
\label{sect:feassoln}

In the above we have presented several approaches for finding lower
bounds for $\cut(m)$. In addition, we have found
matrices $X$ that approximate the bound and satisfy some of the graph
partitioning constraints. Specifically, we obtain two approximate
solutions $X_A,X_L \in \EE$ in \eqref{eq:XattainAL}, an approximate solution to \eqref{eq:qplb2} which
can be transformed into an $n\times k$ matrix via \eqref{eq:XZrelate2}, and the $X_1$, $X_2$ described in
Proposition~\ref{lem:linequals}.
We now use these to obtain feasible solutions (partition matrices)
and thus obtain upper bounds.

We show below that we can find the closest feasible partition
matrix $X$ to a given approximate matrix $\bar X$ using linear programming, where $\bar X$ is
found, for example, using the projected eigenvalue, QP or SDP lower bounds.
Note that \eqref{eq:upperbdLP} is a \emph{transportation problem} and
therefore the optimal $X$ in \eqref{eq:upperbdLP} can be found in
strongly polynomial time $(O(n^2))$, see
e.g.,~\cite{MR861043,MR1159330}.

\begin{theorem}
Let $\bar X \in \EE$ be given. Then the closest partition matrix $X$ to
$\bar X$ in Fr\"{o}benius norm can be found by using the simplex method to
solve the linear program
\begin{equation}\label{eq:upperbdLP}
\begin{array}{ccc}
\min & -\trace \bar X^T X  \\
{\rm s.t.}   &  Xe = e,  \\
      &  X^T e = m,  \\
      & X \geq 0.
\end{array}
\end{equation}
\end{theorem}
\begin{proof}
Observe that for any partition matrix $X$, $\trace X^TX = n$. Hence, we have
\[
\min_{X\in \Mcal_m}\|\bar X - X\|^2_F = \trace(\bar X^T\bar X)+n
           +2 \min_{X\in \Mcal_m}\trace \left(-\bar X^T X\right).
\]
The result now follows from this and the fact that $\Mcal_m = {\rm ext}(\DD)$, as stated in \eqref{eq:partmatrices}. (This is similar to what is done in
\cite{KaReWoZh:94}.)
\end{proof}

\section{Numerical Tests}
\label{sect:numertests}

In this section, we provide empirical comparisons for
 the lower and upper bounds presented above. All the numerical tests are
performed in MATLAB version R2012a on a \emph{single} node of the
\href{http://www.math.uwaterloo.ca/~rblander/cops/}
{\emph{COPS}} cluster 
at University of Waterloo.
It is an SGI XE340 system, with two 2.4 GHz quad-core Intel E5620 Xeon 64-bit CPUs and 48 GB RAM,
equipped with SUSE Linux Enterprise server 11 SP1.

\subsection{Random Tests with Various Sizes}\label{sect:numertestssmall}

In this subsection, we compare the bounds on two kinds of randomly generated graphs of various sizes:
\begin{enumerate}
  \item Structured graphs: These are formed by first generating
$k$ disjoint cliques (of sizes $m_1,\ldots,m_k$, randomly chosen from $\{2,...,{\sf imax}+1\}$).
 We join the first $k-1$ cliques to every node of the $k$th clique.
We then add $u_0$ edges between the
first $k-1$ cliques, chosen uniformly at random from the complement graph.
In our tests, we set $u_0 = \lfloor e_cp\rfloor$, where $e_c$ is the number of
edges in the complement graph and $0\le p<1$. By construction, $u_0 \ge {\rm cut}(m)$.
  \item Random graphs: We start by fixing positive integers  $k,{\sf imax}$ and
generating integers $m_1,\ldots,m_k$, each chosen randomly from $\{2,...,{\sf imax}+1\}$.
    We generate a graph with $n = e^Tm$ nodes.
   The incidence matrix is generated with the MATLAB command:
   \begin{verbatim}
     A = round(rand(n)); A = round((A + A')/2); A = A - diag(diag(A));\end{verbatim}
Consequently, an edge is chosen with probability $0.75$.
\end{enumerate}

First, we note the following about the eigenvalue bounds.
Figures~\ref{fig:neggamma} and \ref{fig:posgamma}
\begin{figure}[h!]
\centering
    \caption{Negative value for optimal $\gamma$}
    \label{fig:neggamma}
    \includegraphics[scale = 0.5]{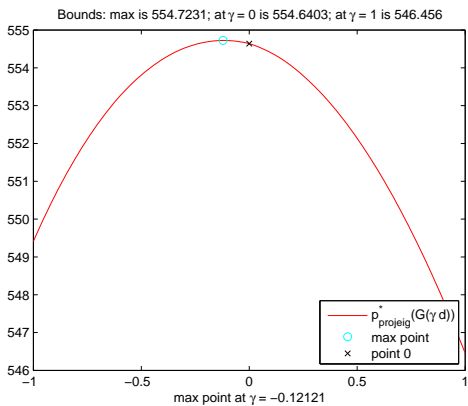}
\end{figure}
\begin{figure}[h!]
\centering
    \caption{Positive value for optimal $\gamma$}
    \label{fig:posgamma}
    \includegraphics[scale = 0.5]{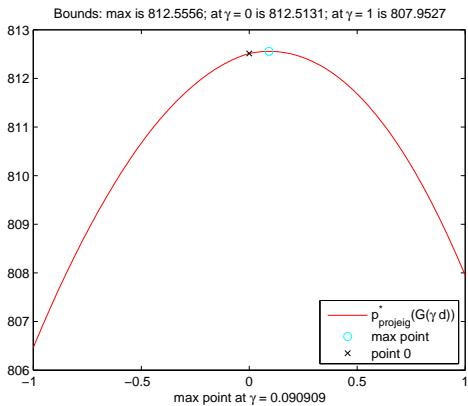}
\end{figure}
show the difference in the projected eigenvalue
bounds from using $A-\gamma \Diag(d)$ for a random $d\in \Rn$ on two structured graphs.
This is typical of what we saw in our tests, i.e. that the maximum bound
is near $\gamma=0$. We had similar results for the specific choice $d=Ae$.
This empirically suggests that using $A$ would yield a better projected eigenvalue lower bound.
This phenomenon will also be observed in subsequent tests.

In Tables~\ref{tables1} and \ref{tables2}, we consider small instances where
$k = 4$, $5$, $p = 20\%$ and ${\sf imax} = 10$. We consider the
projected eigenvalue bounds with $G = -L$ (${\rm eig}_{-L}$) and
$G = A$ (${\rm eig}_{A}$), the QP bound with $G = A$,
the SDP bound and the doubly nonnegative programming (DNN) bound.\footnote{The doubly nonnegative programming relaxation is obtained by imposing the constraint $\widehat V Z \widehat V^T\ge 0$ onto \eqref{eq:SDPfinal}. Like the SDP relaxation, the bound obtained from this approach is independent of $d$. In
our implementation, we picked $G = A$ for both the SDP and the DNN bounds.} For each approach, we present the lower bounds (rounded up to the nearest integer) and the corresponding upper bounds (rounded down to the nearest integer) obtained via the technique described in Section~\ref{sect:feassoln}.\footnote{The SDP and DNN problems are solved via SDPT3 (version 4.0),
\cite{MR1976479}, with tolerance {\sf gaptol} set to be $1e{-6}$ and
$1e{-3}$ respectively. The problems \eqref{eq:LP} and \eqref{eq:qplb2}
are solved via SDPT3 (version 4.0) called by CVX (version 1.22),
\cite{cvx}, using the default settings. The problem \eqref{eq:upperbdLP} is solved using simplex method in MATLAB, again using the default settings.
} We also present the relative gap (Rel. gap), defined as
\begin{equation}\label{eq:relgap}
  \mbox{Rel. gap} = \frac{\mbox{best upper bound} - \mbox{best lower bound}}{\mbox{best upper bound} + \mbox{best lower bound}}.
\end{equation}
In terms of lower bounds, the DNN approach usually gives the best lower bounds.
While the SDP approach gives better lower bounds than the QP
approach for random graphs, they are comparable for structured graphs.
Moreover, the projected eigenvalue lower bounds with $A$ always
outperforms the ones with $-L$.  On the other hand,
the DNN approach usually gives the best upper bounds.

\begin{table}[h]
\small
\begin{center}
\begin{tabular}{|r|r|r|r||r|r|r|r|r||r|r|r|r|r||r|}  \hline
\multicolumn{4}{|c||}{Data} & \multicolumn{5}{c||}{Lower bounds} & \multicolumn{5}{c||}{Upper bounds} & \multicolumn{1}{c|}{Rel. gap}
\\ 
$n$ & $k$ & $|E|$ & $u_0$ & ${\rm eig}_{-L}$ & ${\rm eig}_{A}$ & QP & SDP & DNN & ${\rm eig}_{-L}$ & ${\rm eig}_{A}$ & QP & SDP & DNN &
\\ \hline
 31 & 4 & 362 & 25 & 21 & 22 & 24 & 23 & 25 & 68  & 102 & 25  & 36  & 25  & 0.0000 \\
 18 & 4 & 86  & 16 & 13 & 14 & 15 & 16 & 16 & 22  & 35  & 16  & 19  & 16  & 0.0000 \\
 29 & 5 & 229 & 44 & 32 & 37 & 40 & 39 & 44 & 76  & 74  & 44  & 53  & 44  & 0.0000 \\
 41 & 5 & 453 & 91 & 76 & 84 & 86 & 86 & 91 & 159 & 162 & 101 & 125 & 102 & 0.0521 \\
\hline
\end{tabular}
\end{center}
\caption{Results for small structured graphs}\label{tables1} \normalsize
\end{table}

\begin{table}[h]
\small
\begin{center}
\begin{tabular}{|r|r|r||r|r|r|r|r||r|r|r|r|r||r|}  \hline
\multicolumn{3}{|c||}{Data} & \multicolumn{5}{c||}{Lower bounds} & \multicolumn{5}{c||}{Upper bounds} & \multicolumn{1}{c|}{Rel. gap}
\\ 
$n$ & $k$ & $|E|$ & ${\rm eig}_{-L}$ & ${\rm eig}_{A}$ & QP & SDP & DNN & ${\rm eig}_{-L}$ & ${\rm eig}_{A}$ & QP & SDP & DNN &
\\ \hline
 25 & 4 & 231 & 53  & 59  & 64  & 67  & 71  & 80  & 79  & 74  & 75  & 72  & 0.0070  \\
 23 & 4 & 189 & 7   & 9   & 12  & 14  & 18  & 25  & 24  & 22  & 22  & 20  & 0.0526  \\
 32 & 5 & 379 & 101 & 112 & 119 & 123 & 134 & 152 & 151 & 141 & 141 & 137 & 0.0111  \\
 28 & 5 & 266 & 77  & 89  & 95  & 100 & 106 & 124 & 132 & 111 & 115 & 112 & 0.0230  \\
\hline
\end{tabular}
\end{center}
\caption{Results for small random graphs}\label{tables2} \normalsize
\end{table}

We consider medium-sized instances in Tables~\ref{table1} and \ref{table2},
where $k = 8$, $10$, $12$, $p = 20\%$ and ${\sf imax} = 20$.
We do not consider DNN bounds due to computational complexity.
We see that the lower bounds always satisfy
${\rm eig}_{-L} \leq {\rm eig}_{A} \leq {\rm QP}$. In particular,
we note that the
(lower) projected eigenvalue bounds with $A$ always outperform the ones
with $-L$. However, what is surprising is that the lower
projected eigenvalue bound
with $A$ (for structured graphs) sometimes outperforms the SDP lower
bound. This illustrates the strength of the heuristic that replaces the
quadratic objective function with the sum of a quadratic and linear term
and then solves the linear part exactly over the partition matrices.

\begin{table}[h!]
\small
\begin{center}
\begin{tabular}{|r|r|r|r||r|r|r|r||r|r|r|r||r|}  \hline
\multicolumn{4}{|c||}{Data} & \multicolumn{4}{c||}{Lower bounds} & \multicolumn{4}{c||}{Upper bounds} & \multicolumn{1}{c|}{Rel. gap}
\\ 
$n$ & $k$ & $|E|$ & $u_0$ & ${\rm eig}_{-L}$ & ${\rm eig}_{A}$ & QP & SDP & ${\rm eig}_{-L}$ & ${\rm eig}_{A}$ & QP & SDP &
\\ \hline
 69  & 8  & 1077 & 317  & 249  & 283  & 290  & 281  & 516  & 635  & 328  & 438  & 0.0615    \\
 114 & 8  & 3104 & 834  & 723  & 785  & 794  & 758  & 1475 & 1813 & 834  & 1099 & 0.0246   \\
 85  & 8  & 2164 & 351  & 262  & 319  & 327  & 320  & 809  & 384  & 367  & 446  & 0.0576    \\
 116 & 10 & 3511 & 789  & 659  & 725  & 737  & 690  & 1269 & 2035 & 796  & 1135 & 0.0385    \\
 104 & 10 & 2934 & 605  & 500  & 546  & 554  & 529  & 1028 & 646  & 631  & 836  & 0.0650   \\
 78  & 10 & 1179 & 455  & 358  & 402  & 413  & 389  & 708  & 625  & 494  & 634  & 0.0893    \\
 129 & 12 & 3928 & 1082 & 879  & 988  & 1001 & 965  & 1994 & 1229 & 1233 & 1440 & 0.1022    \\
 120 & 12 & 3102 & 1009 & 833  & 913  & 926  & 893  & 1627 & 1278 & 1084 & 1379 & 0.0786    \\
 126 & 12 & 2654 & 1305 & 1049 & 1195 & 1218 & 1186 & 1767 & 1617 & 1361 & 1736 & 0.0554    \\
\hline
\end{tabular}
\end{center}
\caption{Results for medium-sized structured graphs}\label{table1} \normalsize
\end{table}

\begin{table}[h!]
\small
\begin{center}
\begin{tabular}{|r|r|r||r|r|r|r||r|r|r|r||r|}  \hline
\multicolumn{3}{|c||}{Data} & \multicolumn{4}{c||}{Lower bounds} & \multicolumn{4}{c||}{Upper bounds} & \multicolumn{1}{c|}{Rel. gap}
\\ 
$n$ & $k$ & $|E|$ & ${\rm eig}_{-L}$ & ${\rm eig}_{A}$ & QP & SDP & ${\rm eig}_{-L}$ & ${\rm eig}_{A}$ & QP & SDP &
\\ \hline
 96  & 8  & 3405  & 1982 & 2103 & 2126 & 2146 & 2357 & 2353 & 2354 & 2368 & 0.0460  \\
 96  & 8  & 3403  & 2264 & 2420 & 2439 & 2451 & 2668 & 2652 & 2658 & 2696 & 0.0394  \\
 94  & 8  & 3292  & 1795 & 1885 & 1910 & 1930 & 2128 & 2141 & 2092 & 2130 & 0.0403  \\
 90  & 10 & 3009  & 1533 & 1622 & 1649 & 1659 & 1867 & 1886 & 1850 & 1873 & 0.0544  \\
 114 & 10 & 4823  & 2218 & 2394 & 2443 & 2459 & 2759 & 2780 & 2725 & 2777 & 0.0513  \\
 110 & 10 & 4542  & 3021 & 3160 & 3185 & 3201 & 3487 & 3491 & 3484 & 3492 & 0.0423  \\
 168 & 12 & 10502 & 7523 & 7860 & 7894 & 7912 & 8509 & 8504 & 8494 & 8594 & 0.0355  \\
 126 & 12 & 5930  & 4052 & 4292 & 4318 & 4330 & 4706 & 4687 & 4672 & 4735 & 0.0380  \\
 134 & 12 & 6616  & 4402 & 4523 & 4557 & 4577 & 4955 & 5004 & 4963 & 5011 & 0.0397  \\
\hline
\end{tabular}
\end{center}
\caption{Results for medium-sized random graphs}\label{table2} \normalsize
\end{table}

In Tables~\ref{tablem1} and \ref{tablem2}, we consider larger
instances with $k = 35$, $45$, $55$, $p = 20\%$ and ${\sf imax} = 100$.
We do not consider SDP and DNN bounds due to computational complexity. We see again that the projected eigenvalue lower bounds with $A$ always
outperforms the ones with $-L$.

\begin{table}[h!]
\small
\begin{center}
\begin{tabular}{|r|r|r|r||r|r||r|r||r|}  \hline
\multicolumn{4}{|c||}{Data} & \multicolumn{2}{c||}{Lower bounds} & \multicolumn{2}{c||}{Upper bounds} & \multicolumn{1}{c|}{Rel. gap}
\\ 
$n$ & $k$ & $|E|$ & $u_0$ & ${\rm eig}_{-L}$ & ${\rm eig}_{A}$ & ${\rm eig}_{-L}$ & ${\rm eig}_{A}$  &
\\ \hline
 2012 & 35 & 575078  & 361996 & 345251 & 356064 & 442567 & 377016 & 0.0286\\
 1545 & 35 & 351238  & 210375 & 193295 & 205921 & 258085 & 219868 & 0.0328\\
 1840 & 35 & 439852  & 313006 & 295171 & 307139 & 371207 & 375468 & 0.0944\\
 1960 & 45 & 532464  & 346838 & 323526 & 339707 & 402685 & 355098 & 0.0222\\
 2059 & 45 & 543331  & 393845 & 369313 & 386154 & 469219 & 483654 & 0.0971\\
 2175 & 45 & 684405  & 419955 & 396363 & 412225 & 541037 & 581416 & 0.1351\\
 2658 & 55 & 924962  & 651547 & 614044 & 638827 & 780106 & 665760 & 0.0206\\
 2784 & 55 & 1063828 & 702526 & 664269 & 690186 & 853750 & 922492 & 0.1059\\
 2569 & 55 & 799319  & 624819 & 586527 & 612605 & 721033 & 713355 & 0.0760\\
\hline
\end{tabular}
\end{center}
\caption{Results for larger structured graphs}\label{tablem1} \normalsize
\end{table}

\begin{table}[h!]
\small
\begin{center}
\begin{tabular}{|r|r|r||r|r||r|r||r|}  \hline
\multicolumn{3}{|c||}{Data} & \multicolumn{2}{c||}{Lower bounds} & \multicolumn{2}{c||}{Upper bounds} & \multicolumn{1}{c|}{Rel. gap}
\\ 
$n$ & $k$ & $|E|$ & ${\rm eig}_{-L}$ & ${\rm eig}_{A}$ & ${\rm eig}_{-L}$ & ${\rm eig}_{A}$ &
\\ \hline
 1608 & 35 & 969450  & 837200  & 851686  & 875955  & 875521  & 0.0138 \\
 1827 & 35 & 1250683 & 1066083 & 1083048 & 1112377 & 1112523 & 0.0134 \\
 1759 & 35 & 1159454 & 1032413 & 1048350 & 1075600 & 1074945 & 0.0125 \\
 2250 & 45 & 1897480 & 1669309 & 1694456 & 1735583 & 1734965 & 0.0118 \\
 2287 & 45 & 1959760 & 1808192 & 1838114 & 1879230 & 1877722 & 0.0107 \\
 2594 & 45 & 2522071 & 2183560 & 2212241 & 2263249 & 2264242 & 0.0114 \\
 2660 & 55 & 2651856 & 2481928 & 2516160 & 2568521 & 2566434 & 0.0099 \\
 2715 & 55 & 2763486 & 2503729 & 2535541 & 2589999 & 2589202 & 0.0105 \\
 2661 & 55 & 2652743 & 2413321 & 2442960 & 2495530 & 2495115 & 0.0106 \\
\hline
\end{tabular}
\end{center}
\caption{Results for larger random graphs}\label{tablem2} \normalsize
\end{table}

We now briefly comment on the computational time (measured by MATLAB tic-toc function) for the above tests. For lower bounds,
the eigenvalue bounds are fastest to compute. Computational time for small, medium and larger problems are usually
less than 0.01 seconds, 0.1 seconds and 0.5 minutes, respectively. The QP bounds are more expensive to compute, taking around $0.5$ to $2$ seconds
for small instances and $0.5$ to $15$ minutes for medium-sized instances.
The SDP bounds are even more expensive to compute, taking $0.5$ to $3$ seconds for small instances
and $2$ minutes to $2$ hours for medium-sized instances. The DNN bounds are the most expensive to compute. Even for small instances,
it can take $20$ seconds to $40$ minutes to compute a bound. For upper bounds, using the MATLAB simplex method,
the time for solving \eqref{eq:upperbdLP} is usually less than 1 second for small and
medium-sized problems; while for the larger problems in Tables~\ref{tablem1} and \ref{tablem2}, it takes $1$ to $5$ minutes.

\paragraph{Finding a Vertex Separator.} Before ending this subsection, we comment on how the above bounds can possibly be used in finding vertex separators when $m$ is not explicitly known beforehand. Since there can be at most $\binom{n-1}{k-1}$ $k$-tuples of integers summing up to $n$, theoretically, one can consider all possible such $m$ and estimate the corresponding $\cut(m)$ with the bounds above.

As an illustration, we consider a concrete instance of a structured graph, generated with $n = 600$, $m_1=m_2=m_3=200$ and $p = 0$. Thus, we have $k=3$, and, by construction, $\cut(m) = 0$.

Suppose that the correct size vector $m$ is not known in advance. Therefore we
now consider a range of estimated vectors $m'$. In Table~\ref{tablenew},
we consider sizes $m'_1$ and $m'_2$ with values taken between $180$ to $220$,
with $m'_3=600-m'_1-m'_2$. We report on the eigenvalue bounds,
the QP bounds and
the SDP bounds for each $m'$. Observe that the SDP lower bounds are usually the largest while the QP upper bounds are usually the smallest. The existence of a vertex separator when $m_1=m_2=m_3=200$ is identified by the QP and SDP bounds.\footnote{The QP lower bound of $1$ in this case actually corresponds to an objective value in the order of $1e-5$. We obtain the $1$ since we always truncate the lower bound to the smallest integer exceeding it.} Furthermore, the QP upper bound being zero for the cases $(m'_1,m'_2)=(180,180)$, $(180,200)$ or $(200,180)$ also indicates the existence of a vertex separator.

\begin{table}[h!]
\small
\begin{center}
\begin{tabular}{|r|r||r|r|r|r||r|r|r|r|}  \hline
\multicolumn{2}{|c||}{Data} & \multicolumn{4}{c||}{Lower bounds} & \multicolumn{4}{c|}{Upper bounds}
\\ 
$m'_1$ & $m'_2$ & ${\rm eig}_{-L}$ & ${\rm eig}_{A}$ & QP & SDP & ${\rm eig}_{-L}$ & ${\rm eig}_{A}$ & QP & SDP \\ \hline
 180 & 180 & -3600 & -2400 & -2400 & -1800 & 2520 & 32400 &    0 &  540 \\
 180 & 200 & -1922 & -1281 & -1270 & -949 & 2538 & 36000 &    0 & 3240 \\
 180 & 220 &  -99 &  -66 &  -16 &   0 & 3600 & 39600 & 3600 & 4312 \\
 200 & 180 & -1922 & -1281 & -1270 & -949 & 2538 & 36000 &    0 & 1440 \\
 200 & 200 &   0 &   0 &    1 &   0 & 2200 & 39801 &    0 &    0 \\
 200 & 220 & 2074 & 2716 & 2759 & 4000 & 4000 & 40000 & 4398 & 11832 \\
 220 & 180 &  -99 &  -66 &  -16 &   0 & 3600 & 39600 & 3958 & 19768 \\
 220 & 200 & 2074 & 2716 & 2759 & 4000 & 4000 & 40000 & 11518 & 11200 \\
 220 & 220 & 4400 & 5867 & 5867 & 8400 & 8400 & 40241 & 8400 & 12916 \\
\hline
\end{tabular}
\end{center}
\caption{Results for medium-sized graph without an explicitly known $m$}\label{tablenew} \normalsize
\end{table}

\subsection{Large Sparse Projected Eigenvalue Bounds}\label{sect:numertestslarge}

We assume that $n \gg k$.  The projected eigenvalue bound in Theorem
\ref{thm:projeig} in \eqref{eq:Aprojeigbnd}
is composed of a constant term, a minimal scalar product of
$k-1$ eigenvalues and a linear term. The constant term and linear term are
trivial to evaluate and essentially take no CPU time. The evaluation of
the $k-1$ eigenvalues of $\widehat B$ is also efficient and accurate as
the matrix is small and symmetric.  The only significant cost is
the evaluation of the largest $k-2$ eigenvalues and the smallest eigenvalue of
$\widehat G$. In our test below, we use $G=A$ for simplicity. This choice is
also justified by our numerical results in the previous subsection and the observation from
Figures~\ref{fig:neggamma} and \ref{fig:posgamma}.

We use the MATLAB
\emph{eigs} command for the $k-1$ eigenvalues of $V^TAV$ for the lower bound.
Since the corresponding \eqref{eq:upperbdLP} has much larger
dimension than we considered in the previous subsection,
we turn to
\href{http://www-01.ibm.com/software/commerce/optimization/cplex-optimizer/}{IBM ILOG CPLEX}
version 12.4 (MATLAB interface)
with default settings to solve for the upper bound.
We use the MATLAB tic-toc function to time the routine for finding the
lower bound, and report {\sf output.time} from
the function {\sf cplexlp.m} as the cputime for finding the upper bound.

We use two different choices $V_0$ and $V_1$
for the matrix $V$ in \eqref{eq:defPQ}.
\begin{enumerate}
\item
We choose the following matrix $V_0$ with mutually orthogonal columns
that satisfies  $V_0^Te=0$.\footnote{Choosing a sparse $V$ in the
orthogonal matrix in \eqref{eq:defPQ} would speed up the calculation
of the eigenvalues. Choosing a sparse $V$ would be easier if $V$ did not
require orthonormal columns but just linearly independent columns,
i.e.,~if we could arrange for a parametrization as in Lemma
\ref{lem:paramE} without $P$ orthogonal.}
\[
V_0=
\begin{bmatrix}
1 & 1 & 1 & \ldots & 1   \cr
-1 & 1 & 1 & \ldots & 1   \cr
0 & -2 & 1 & \ldots & 1   \cr
0 & 0 & -3 & \ldots & 1   \cr
\dots & \dots & \dots & \dots & \cr
0 & 0 &  0 & \ldots &  -(n-1)   \cr
\end{bmatrix}
\]
Let $s=\begin{pmatrix}  \|V_0(:,i)\| \end{pmatrix} \in \R^{n-1}$.
Then the operation needed for the MATLAB large sparse
eigenvalue function \emph{eigs} is  ($*$ denotes multiplication
and $\cdot '$ denotes transpose, $./$ denotes
elementwise division)
\beq
\label{eq:Vbang}
\widehat A*v = V'*(A*(V * v)) =
        V'_0*(A * (V_0 * (v./s) ))  ./s.
\eeq
Thus we never form the matrix $\widehat A$ and we preserve the
structure of $V_0$ and sparsity of $A$
when doing the matrix-vector multiplications.

\item
An alternative approach uses
\[
   V_1 =
\begin{bmatrix}
\begin{bmatrix}
\begin{bmatrix}
I_{\left \lfloor \frac n2 \right \rfloor} \otimes
                                      \frac 1{\sqrt 2} \begin{bmatrix}
                                      1\cr -1
                                      \end{bmatrix}
\end{bmatrix}
\cr 0_{(n-2\left \lfloor \frac n2 \right \rfloor),
          \left \lfloor \frac n2 \right \rfloor}
\end{bmatrix}
\begin{bmatrix}
\begin{bmatrix}
I_{\left \lfloor \frac n4 \right \rfloor} \otimes
                                     \frac 12 \begin{bmatrix}
                         1 \cr 1 \cr -1 \cr -1
                                      \end{bmatrix}
\end{bmatrix}
\cr 0_{(n-4\left \lfloor \frac n4 \right \rfloor),
          \left \lfloor \frac n4 \right \rfloor}
\end{bmatrix}
\begin{bmatrix}
\ldots
\end{bmatrix}
\begin{bmatrix}
\widehat V
\end{bmatrix}
\end{bmatrix}_{n\times n-1}
\]
i.e.,~the block matrix consisting of $t$ blocks formed from Kronecker products
along with one block $\widehat V$ to complete the appropriate size so that $V^TV=I_{n-1}$, $V^Te=0$.
We take advantage of the $0$, $1$ structure of the Kronecker blocks and
delay the scaling factors till the end. Thus we use the same type of
operation as in \eqref{eq:Vbang} but with $V_1$ and the new scaling
vector $s$.
\end{enumerate}

The results on large scale problems using the two
choices $V_0$ and $V_1$ are reported in Tables~\ref{table:large1}, \ref{table:large2} and \ref{table:large3}.
For simplicity, we only consider random graphs, with various ${\sf imax}$ and $k$. We generate $m$ as described before
and use the commands
\begin{verbatim}
             A=sprandsym(n,dens); A(1:n+1:end)=0; A(abs(A)>0)=1;
\end{verbatim}
to generate a random incidence matrix, with ${\rm dens} = 0.05/i$, for $i=1,\ldots,10$.
In the tables, we present the number of nodes, sets, edges ($n$, $k$, $|E|$),
the true density of the random graph $density:=2|E|/(n(n-1))$, the lower and upper projected eigenvalue bounds,
the relative gap \eqref{eq:relgap}, and the cputime (in seconds) for computing the bounds.


The results using the matrix $V_0$ are in Tables~\ref{table:large1}.
Here the cost for
finding the lower bound using the eigenvalues
becomes significantly higher than the cost for finding the upper bound
using the simplex method.

\begin{table}[h!]
\small
\begin{center}
\begin{tabular}{|r|r|r||r|r|r|r||r|r|r|r|}  \hline
$n$ & $k$ & $|E|$ & density &  lower  & upper
  &   Rel. gap &  cpu (low) &  cpu (up)
\\ \hline
$13685$ & $68$ & $4566914$ & $4.88\times 10^{-2}$ & $3958917$ & $4271928$ & $0.0380$ & $409.4$ & $  7.1$ \\
$13599$ & $65$ & $2282939$ & $2.47\times 10^{-2}$ & $1967979$ & $2181778$ & $0.0515$ & $330.1$ & $  6.1$ \\
$13795$ & $68$ & $1572487$ & $1.65\times 10^{-2}$ & $1314033$ & $1495421$ & $0.0646$ & $316.2$ & $  7.9$ \\
$13249$ & $66$ & $1090447$ & $1.24\times 10^{-2}$ & $832027$ & $985375$ & $0.0844$ & $265.6$ & $  7.4$ \\
$12425$ & $66$ & $767961$ & $9.95\times 10^{-3}$ & $589226$ & $710093$ & $0.0930$ & $253.2$ & $  6.0$ \\
$13913$ & $66$ & $803074$ & $8.30\times 10^{-3}$ & $591486$ & $726783$ & $0.1026$ & $304.9$ & $  7.1$ \\
$14144$ & $65$ & $711936$ & $7.12\times 10^{-3}$ & $543017$ & $666721$ & $0.1023$ & $274.4$ & $  7.1$ \\
$13667$ & $67$ & $581930$ & $6.23\times 10^{-3}$ & $427464$ & $538291$ & $0.1148$ & $254.9$ & $  6.5$ \\
$12821$ & $68$ & $455329$ & $5.54\times 10^{-3}$ & $329902$ & $422417$ & $0.1230$ & $244.5$ & $  7.4$ \\
$12191$ & $69$ & $370595$ & $4.99\times 10^{-3}$ & $262521$ & $343426$ & $0.1335$ & $211.1$ & $  6.3$
\\ \hline
\end{tabular}
\end{center}
\caption{Large scale random graphs; {\sf imax} $400$; $k\in [65,70]$, using $V_0$}
\label{table:large1} \normalsize
\end{table}

The results using the matrix $V_1$ are shown in Tables~\ref{table:large2} and \ref{table:large3}.
We can see the obvious improvement in cputime
when finding the lower bounds using $V_1$ compared to using $V_0$, which becomes more significant when the graph gets sparser.

\begin{table}[h!]
\small
\begin{center}
\begin{tabular}{|r|r|r||r|r|r|r||r|r|r|r|}  \hline
$n$ & $k$ & $|E|$ & density &  lower  & upper
  &   Rel. gap &  cpu (low) &  cpu (up)
\\ \hline
$14680$ & $69$ & $5254939$ & $4.88\times 10^{-2}$ & $4586083$ & $4955524$ & $0.0387$ & $262.9$ & $  6.4$ \\
$14464$ & $65$ & $2583109$ & $2.47\times 10^{-2}$ & $2133187$ & $2397098$ & $0.0583$ & $135.5$ & $  6.0$ \\
$14974$ & $69$ & $1852955$ & $1.65\times 10^{-2}$ & $1555718$ & $1776249$ & $0.0662$ & $ 98.2$ & $  6.9$ \\
$13769$ & $65$ & $1177579$ & $1.24\times 10^{-2}$ & $956260$ & $1124729$ & $0.0810$ & $ 44.4$ & $  5.9$ \\
$13852$ & $69$ & $954632$ & $9.95\times 10^{-3}$ & $775437$ & $924265$ & $0.0876$ & $ 51.3$ & $  6.0$ \\
$12516$ & $65$ & $650028$ & $8.30\times 10^{-3}$ & $475477$ & $598372$ & $0.1144$ & $ 34.0$ & $  4.3$ \\
$13525$ & $66$ & $651025$ & $7.12\times 10^{-3}$ & $508512$ & $630663$ & $0.1072$ & $ 33.3$ & $  5.8$ \\
$13622$ & $66$ & $578111$ & $6.23\times 10^{-3}$ & $414786$ & $535755$ & $0.1273$ & $ 34.6$ & $  6.0$ \\
$13004$ & $65$ & $468437$ & $5.54\times 10^{-3}$ & $328925$ & $434795$ & $0.1386$ & $ 29.1$ & $  5.2$ \\
$14659$ & $69$ & $535899$ & $4.99\times 10^{-3}$ & $380571$ & $501082$ & $0.1367$ & $ 27.2$ & $  5.9$
\\ \hline
\end{tabular}
\end{center}
\caption{Large scale random graphs; {\sf imax} $400$; $k\in [65,70]$, using $V_1$}
\label{table:large2} \normalsize
\end{table}
\begin{table}[h!]
\small
\begin{center}
\begin{tabular}{|r|r|r||r|r|r|r||r|r|r|r|}  \hline
$n$ & $k$ & $|E|$ & density &  lower  & upper
  &   Rel. gap &  cpu (low) &  cpu (up)
\\ \hline
$22840$ & $80$ & $12721604$ & $4.88\times 10^{-2}$ & $11548587$ & $12262688$ & $0.0300$ & $782.4$ & $ 12.5$ \\
$16076$ & $77$ & $3190788$ & $2.47\times 10^{-2}$ & $2754650$ & $3053622$ & $0.0515$ & $199.1$ & $  8.9$ \\
$20635$ & $77$ & $3519170$ & $1.65\times 10^{-2}$ & $2916188$ & $3287657$ & $0.0599$ & $228.5$ & $ 10.1$ \\
$19408$ & $79$ & $2339682$ & $1.24\times 10^{-2}$ & $1989278$ & $2272340$ & $0.0664$ & $147.3$ & $ 10.6$ \\
$17572$ & $76$ & $1536161$ & $9.95\times 10^{-3}$ & $1188933$ & $1417085$ & $0.0875$ & $ 83.6$ & $  9.0$ \\
$18211$ & $80$ & $1376087$ & $8.30\times 10^{-3}$ & $1127696$ & $1336407$ & $0.0847$ & $ 90.7$ & $ 11.2$ \\
$21041$ & $80$ & $1575333$ & $7.12\times 10^{-3}$ & $1232501$ & $1482463$ & $0.0921$ & $ 93.6$ & $ 10.5$ \\
$20661$ & $77$ & $1329856$ & $6.23\times 10^{-3}$ & $1023056$ & $1251437$ & $0.1004$ & $ 74.5$ & $ 11.8$ \\
$19967$ & $77$ & $1104350$ & $5.54\times 10^{-3}$ & $831335$ & $1035126$ & $0.1092$ & $ 74.0$ & $  9.6$ \\
$20839$ & $78$ & $1082982$ & $4.99\times 10^{-3}$ & $831672$ & $1034104$ & $0.1085$ & $ 73.9$ & $ 11.0$
\\ \hline
\end{tabular}
\end{center}
\caption{Large scale random graphs; {\sf imax} $500$; $k\in [75,80]$, using $V_1$}
\label{table:large3} \normalsize
\end{table}

In all three tables, we note that the relative gaps deteriorate as the density decreases.
Also, the cputime for the eigenvalue bound is
significantly better when using $V_1$ suggesting that sparsity of
$V_1$ is better exploited in the MATLAB \emph{eigs} command.

\section{Conclusion}
\label{sect:concl}
In this paper, we presented
eigenvalue, projected eigenvalue, QP,  and SDP
lower and upper bounds for a minimum cut problem.
In particular, we looked at a variant of the projected eigenvalue bound found in \cite{ReLiPi:13} and showed numerically that
our variant is stronger.
We also proposed a new QP bound following the approach in \cite{AnsBrix:99}, making use of a duality result presented in \cite{PoRe:08}.
In addition, we
studied an SDP relaxation and demonstrated its strength by showing the redundancy of
quadratic (orthogonality) constraints.
We emphasize that these techniques for deriving bounds for our cut minimization problem can be adapted to derive new results for the GP. Specifically,
one can easily adapt our derivation and obtain a QP lower bound for the GP, which was not previously known in the literature.
Our derivation of the simple facially reduced SDP relaxation \eqref{eq:SDPfinal} can also be adapted to simplify the existing
SDP relaxation for the GP studied in \cite{WoZh:96}.

We also compared these bounds numerically on randomly generated graphs of various sizes.
Our numerical tests illustrate that the projected eigenvalue bounds
can be found efficiently for large scale sparse problems
and that they compare well against other more expensive bounds on smaller problems.
It is surprising that the projected eigenvalue bounds using the adjacency matrix $A$ are both cheap to calculate and strong.

\cleardoublepage
\addcontentsline{toc}{section}{Index}
\label{ind:index}
\printindex

\addcontentsline{toc}{section}{Bibliography}
\bibliography{.master,.edm,.psd,.bjorBOOK}

\end{document}